\documentclass{amsart}
\usepackage{amsmath,amsthm,amsfonts,amssymb,latexsym,mathrsfs,graphicx}

\usepackage{hyperref}

\usepackage{enumerate}
\usepackage[shortlabels]{enumitem}
\usepackage{color}

\newtheorem{theorem}{Theorem}[section]
\newtheorem{corollary}[theorem]{Corollary}
\newtheorem{lemma}[theorem]{Lemma}
\newtheorem{proposition}[theorem]{Proposition}
\theoremstyle{definition}

\newtheorem{example}[theorem]{Example}

\newtheorem*{ThmA}{Theorem A}
\newtheorem*{ThmB}{Theorem B}
\newtheorem*{ThmC}{Theorem C}

\newenvironment{enumeratei}{\begin{enumerate}[\upshape (a)]}
    {\end{enumerate}}

\def\irr#1{{\rm Irr}(#1)}
\def\cent#1#2{{\bf C}_{#1}(#2)}

\def\nor{\trianglelefteq\,}

\def\zent#1{{\bf Z}(#1)}
\def\sbs{\subseteq}

\def\aut#1{{\rm Aut}(#1)}
\def\out#1{{\rm Out}(#1)}

\def\fit#1{{\bf F}(#1)}
\def\frat#1{{\bf \Phi}(#1)}

\newcommand{\N}{{\mathbb N}}
\newcommand{\pr}{{\mathbb P}}

\def\irr#1{{\rm Irr}(#1)}

\def\m#1{{\rm m}(#1)}

\def\cent#1#2{{\bf C}_{#1}(#2)}

\def\nor{\trianglelefteq\,}
\def\norm#1#2{{\bf N}_{#1}(#2)}

\def\zent#1{{\bf Z}(#1)}
\def\sbs{\subseteq}

\def\aut#1{{\rm Aut}(#1)}

\def\out#1{{\rm Out}(#1)}

\def\fit#1{{\bf F}(#1)}

\def\fitg#1{{\bf F^*}(#1)}

\def\PSL#1{{\rm PSL}_{2}(#1)}

\def\E#1{{\bf E}(#1)}
\def\b#1{\overline{#1}}
\def\m#1{{\rm m}(#1)}

\def\irr#1{{\rm Irr}(#1)}

\def\cd#1{{\rm cd}(#1)}

\def\cent#1#2{{\bf C}_{#1}(#2)}

\def\zent#1{{\bf Z}(#1)}

\def\norm#1#2{{\bf N}_{#1}(#2)}

\def\sbs{\subseteq}

\def \o#1{\overline{#1}}
\mathchardef\coso="2023

\begin{document}

\title{On Huppert's rho-sigma conjecture}

\author[Z. Akhlaghi et al.]{Zeinab Akhlaghi}
\address{Zeinab Akhlaghi, Faculty of Math. and Computer Sci., \newline Amirkabir University of Technology (Tehran Polytechnic), 15914 Tehran, Iran.\newline
School of Mathematics,
Institute for Research in Fundamental Science(IPM)
P.O. Box:19395-5746, Tehran, Iran.}
\email{z\_akhlaghi@aut.ac.ir}

\author[]{Silvio Dolfi}
\address{Silvio Dolfi, Dipartimento di Matematica e Informatica U. Dini,\newline
Universit\`a degli Studi di Firenze, viale Morgagni 67/a,
50134 Firenze, Italy.}
\email{silvio.dolfi@unifi.it}

\author[]{Emanuele Pacifici}
\address{Emanuele Pacifici, Dipartimento di Matematica F. Enriques,
\newline Universit\`a degli Studi di Milano, via Saldini 50,
20133 Milano, Italy.}
\email{emanuele.pacifici@unimi.it}

\thanks{
The first author was partially supported by a grant from IPM (No. 99200028). The second and the third authors were partially supported by INDAM-GNSAGA}


\begin{abstract} 
For an irreducible complex character \(\chi\) of the finite group \(G\), let \(\pi(\chi)\) denote the set of prime divisors of the degree \(\chi(1)\) of \(\chi\). Denote then by \(\rho(G)\) the union of all the sets \(\pi(\chi)\) and by \(\sigma(G)\) the largest value of \(|\pi(\chi)|\), as \(\chi\) runs in \(\irr G\). The \emph{$\rho$-$\sigma$ conjecture}, formulated by Bertram Huppert in the 80's, predicts that \(|\rho(G)|\leq 3\sigma(G)\) always holds, whereas \(|\rho(G)|\leq 2\sigma(G)\) holds if \(G\) is solvable; moreover, O. Manz and T.R. Wolf proposed a ``strengthened" form of the conjecture in the general case, asking whether \(|\rho(G)|\leq 2\sigma(G)+1\) is true for every finite group \(G\). In this paper we study the strengthened $\rho$-$\sigma$ conjecture \emph{for the class of finite groups having a trivial Fitting subgroup}: in this context, we prove that the conjecture is true provided \(\sigma(G)\leq 5\), but it is false in general if \(\sigma(G)\geq 6\). Instead, we establish that \(|\rho(G)|\leq 3\sigma(G)-4\) holds for every finite group with a trivial Fitting subgroup and with \(\sigma(G)\geq 6\) (this being the right, best possible bound). Also,  we improve the up-to-date best bound for the solvable case, showing that we have \(|\rho(G)|\leq 3\sigma(G)\) whenever \(G\) belongs to one particular class including all the finite solvable groups.    
\end{abstract}

\maketitle
\section{Introduction}
The set \(\cd G=\{\chi(1)\mid\chi\in\irr G\}\) consisting of the degrees of the irreducible complex characters of a finite group \(G\) has been an object of considerable interest since the second part of the \(20^{\rm{th}}\) century, and the study of the arithmetical structure of this set is a particularly intriguing aspect of Character Theory of finite groups (see for instance \cite{L}).  
A remarkable question in this research area was posed by Bertram Huppert in the 80's: is it true that  at least one of the  character degrees  is divisible by a ``large" portion  of the entire set of primes that  appear as divisors of some character degree?
More precisely, denoting by $\pi(n)$  the set of prime divisors of an integer $n$, and writing for short
$\pi(\chi)$ instead of $\pi(\chi(1))$ when $\chi \in \irr G$, one defines
$$\rho(G) = \bigcup_{\chi \in \irr G} \pi(\chi)$$
and
$$ \sigma(G) = \max \{ |\pi(\chi)| \mid \chi \in \irr G \} \;;$$

\noindent Huppert's \emph{$\rho$-$\sigma$ conjecture} predicts that $|\rho(G)| \leq 3 \sigma(G)$  holds for every finite group $G$,
and that $|\rho(G)| \leq 2 \sigma(G)$ if $G$ is solvable.
It is worth noting that the bounds are in some sense  best possible, as they are attained for the groups $A_5$ and $S_4$, respectively.

During the last four decades, several contributions have been given toward the proof of this conjecture.
For solvable groups, the conjecture was proved true by D. Gluck (\cite{G}) for $\sigma(G) \leq 2$, and also in the case that all degrees in $\cd G$ are square-free numbers. The best bound known till now was obtained by O. Manz and T.R. Wolf; they proved in~\cite{MW0} that, if $G$ is solvable, then  $|\rho(G)| \leq 3 \sigma(G) +2$.

As for the non-solvable case, the $\rho$-$\sigma$ conjecture was proved true for all finite non-abelian simple groups by D.L. Alvis and M. Barry (\cite{AB}), whereas the general but weaker bound $|\rho(G)| \leq 7 \sigma(G)$ was obtained by C. Casolo and the second author in~\cite{CD}.

One might wonder whether the factor $3$ is  the right one for non-solvable groups, or one should instead  keep the
factor $2$ and  add a suitable constant for getting a tighter bound. 
In a recent paper (\cite{HTV}), H. Tong-Viet studies the so-called \emph{strengthened}  $\rho$-$\sigma$ conjecture
proposed by Manz and Wolf in~\cite{MW0}, that is $|\rho(G)| \leq 2\sigma(G) + 1$ for every finite group~$G$.
The strengthened $\rho$-$\sigma$ conjecture is  verified in~\cite{HTV} for all finite almost-simple groups  and
also for the groups $G$ such that $\sigma(G) \leq 2$.

In this paper we start by considering finite groups with trivial Fitting subgroup and, for these groups, we establish the strengthened $\rho$-$\sigma$ conjecture whenever $\sigma(G) \leq 5$. But we remark that the strengthened $\rho$-$\sigma$ conjecture is false in general: in fact Example~1.1 of \cite{ADPS}, that is recalled below as Example~\ref{ex},  provides a sequence of finite groups $G$ (with trivial Fitting subgroup) such that the ratio $|\rho(G)|/\sigma(G)$ tends to $3$. 

However, the first main result of this paper also establishes the right bound in this setting for the case \(\sigma(G)\geq 6\).

\begin{ThmA}
Let $G$ be a finite group with trivial Fitting subgroup. Then the following conclusions hold.
\begin{enumeratei}
\item If $\sigma(G) \leq 5$, then $|\rho(G)| \leq 2 \sigma(G) +1$.
\item If $\sigma(G) \geq 6$, then $|\rho(G)| \leq 3 \sigma(G) -4$.
\end{enumeratei}
\end{ThmA}

The groups in Example~\ref{ex} show that the above bounds are sharp.  

\medskip
In the second part of the paper, we focus on solvable groups and we obtain the following improvement of Manz and Wolf's theorem.

\begin{ThmB}
 Let $G$ be a finite group. If $G$ is solvable, then $|\rho(G)| \leq 3 \sigma(G)$. 
\end{ThmB}

Finally, we extend the bound given in Theorem~B to a wider class of groups.
We recall that the generalized Fitting subgroup $\fitg G$ of the finite group \(G\) is the central product of the Fitting subgroup $\fit G$ and the \emph{layer} subgroup ${\bf E}(G)$, which is the group generated by all the \emph{components} of $G$ (see Section~3). It is well known that $\cent G{\fitg G}\leq \fitg G$. Hence, $\cent G{\fitg G} = \zent{\fit G}$ and, when
$\fit G = 1$, $\cent G{\fitg G} = \cent G{{\bf E}(G)} = 1$.
We consider, as a generalization of this setting, the case when the centralizer $\cent G{{\bf E}(G)}$ of the 
layer subgroup of $G$ is \emph{solvable}. The corresponding family of groups hence contains both the family of the groups with trivial Fitting subgroup and the family of the solvable groups. 
\begin{ThmC}
  Let $G$ be a finite group such that $\cent G{\E G}$ is solvable. Then
  $$|\rho(G)| \leq 3\sigma(G) \; .$$ 
\end{ThmC}

\medskip
For the convenience of the reader, and also because the point of view of the present context is not the same as that of \cite{ADPS}, we close this introductory section by recalling Example~1.1 of \cite{ADPS} and some related comments.
\begin{example}
\label{ex}
Let \(\Pi=\{p_1^{f_1},...,p_n^{f_n}\}\) be a set of prime powers where every prime \(p_i\) is larger than \(5\). Assume that, for every \(i\in\{1,...,n\}\), we have \(|\pi(p_i^{f_i}-1)\setminus\{2,3\}|=|\pi(p_i^{f_i}+1)\setminus\{2,3\}|=1\), and assume further that, for distinct \(r\) and \(s\) in \(\{1,...,n\}\), the intersection of the sets \(\{p_r\}\cup\pi(p_r^{2f_r}-1)\) and \(\{p_s\}\cup\pi(p_s^{2f_s}-1)\) is \(\{2,3\}\). 
Now, setting \(G_{\Pi}=\PSL{p_1^{f_1}}\times\cdots\times\PSL{p_n^{f_n}}\) (note that \(\fit{G_{\Pi}}=1\)) and taking into account that, for \(p>5\) and \(f\geq 1\), we have
$$ \cd{\PSL{p^{f}}} = \{1, p^{f} -1, p^{f} , p^{f}  +1, \frac{1}{2}(p^{f}  + \epsilon)\} \text{ where }
  \epsilon = (-1)^{\frac{p^{f} -1}{2}}$$ (see for instance \cite[Theorem~3.2]{W}), it is easy to see that \(|\rho(G_{\Pi})|=3n+2\) and \(\sigma(G_{\Pi})=n+2\), thus \(|\rho(G_{\Pi})|=3\sigma(G_{\Pi})-4\). As a consequence, if \(n\geq 4\), the strengthened \(\rho\)-\(\sigma\) conjecture as formulated by Manz and Wolf does not hold for the group \(G_{\Pi}\).

Note that \(\Pi=\{29, 67, 157, 227\}\) is a set of four prime powers (in fact, of four primes) satisfying the above conditions. This provides a counterexample to the strengthened \(\rho\)-\(\sigma\) conjecture in which the size of \(\rho(G_{\Pi})\) is \(14\), whereas \(\sigma(G_{\Pi})\) is \(6\).
\end{example}

Let $\Pi_n$ be one particular set of given  size $n$ as in Example~\ref{ex}.
Assuming that such a set exists for arbitrarily large \(n \in \mathbb{N} \), we see that the ratio \(|\rho(G_{\Pi_n})|/\sigma(G_{\Pi_n})\) converges to \(3\) as \(n\) tends to infinity (actually, machine computation with prime numbers up to \(10^6\) enables us to construct a group \(G_{\Pi}\) as in Example~\ref{ex} for which \(|\rho(G_{\Pi})|/\sigma(G_{\Pi})>2.999\)); this leads us to conjecture that, for every positive real number \(\epsilon\), there exists a group \(G\) (with trivial Fitting subgroup and) with \(|\rho(G)|/\sigma(G)>3-\epsilon\). In fact, in \cite{EG} the authors estimate the asymptotic density of one particular set of primes, whose infinitude yields the existence of a set \(\Pi_n\) as above for every \(n\in\N\). This set  is actually infinite, if a generalized form of the Hardy-Littlewood conjecture is assumed (\cite[Theorem 2.3]{EG}) 

 Finally, we mention that every group considered throughout the following discussion is tacitly assumed to be a finite group.
 \section{Preliminaries}

If a group $G$ acts on a set $\Omega$, and $\Delta \subseteq \Omega$, we denote by $G_{\Delta}=\{g\in G \mid\Delta g=\Delta\}$ the
stabilizer of $\Delta$ in $G$. Also, if $n$, $m$ are non-negative integers, we denote by $\pi_{\geq m}(n)$ the set of all prime divisors of $n$ which are greater than or equal
to $m$, whereas \(\pr_m\) will denote the set of all primes in \(\N\) that are smaller than or equal to \(m\). We write $\pi(G)$  and $\pi_{\geq m}(G)$ for  $\pi(|G|)$ and  $\pi_{\geq m}(|G|)$, respectively. Similarly, if $H$ is a subgroup of the group $G$, we use the notation $\pi(G:H)$ and $\pi_{\geq m}(G:H)$.
% for $\pi(|G:H|)$ and $\pi_{\geq m}(|G:H|)$, respectively.
Finally,
for $\chi \in \irr{G}$, as already mentioned we will write  $\pi(\chi)$ in place of  $\pi(\chi(1))$.  

For a group $G$, we define $\m G$ to be the largest integer $ m\geq 5$ 
such that $G$ has a section isomorphic to the alternating group $A_m$ (that is, such that there are subgroups
$K \unlhd H \leq G$ with $H/K \cong A_m$), and we set $\m G = 0 $ if there is no such section in $G$; note that \(\pr_{\m G}\) is contained in \(\pi(G)\), since it is in fact contained in \(\pi(A_m)\). We prove next a straightforward property of \(\m G\).

\begin{lemma}\label{lm}
  Let $G$ be a group, and $N$ a normal subgroup of $G$. Then, setting \(m=\m G\), \(m_1=\m N\) and \(m_2=\m{G/N}\), we have $m = \max \{m_1,m_2\}$. 
  %Furthermore, the set \(\pi_{\leq m}(G)\) lies either in \(\pi_{\leq m_1}(N)\) or in \(\pi_{\leq m_2}(G/N)\).
\end{lemma}
\begin{proof} 
  If \(m=0\), then clearly \(m_1\) and \(m_2\) are both \(0\) as well, %moreover, \(\pi_{\leq m}(G)\) is the empty set, so there is nothing to prove in this case. 
and it is also clear that \(m\geq\max \{m_1,m_2\}\) holds for positive values of \(m\) as well. Assume now $m \geq 5$, consider a section $H/K$ of $G$ such that $H/K \cong A_m$, and set $\b{G} = G/N$. Then the section $\b H / \b K$ of $\b G$ is isomorphic either to $A_m$ or to the trivial group, and we have $m = m_2$ in the former case, whereas \(m=m_1\) in the latter case. 
\end{proof}

The following proposition concerning permutation groups turns out to be very useful in our arguments.

\begin{proposition}\label{permutation}
Let $G$ be a permutation group on the finite set $\Omega$ and $m=\m G$. Then there exist \(\Gamma,\;\Delta\subseteq\Omega\) such that $\Gamma\cap\Delta=\emptyset$ and
\begin{enumeratei}
\item \(\pi_{\geq m}(G)=\pi_{\geq m}(G:G_{\Gamma}\cap G_{\Delta})\),
\item  \(|\pr_m|\leq 2|\pi(G:G_{\Gamma}\cap G_{\Delta})\cap\pr_m|.\)
\end{enumeratei}
\end{proposition}

\begin{proof}
Observe that, under the additional assumption that the action of \(G\) on \(\Omega\) is transitive, our statement is precisely Proposition~1 of \cite{CD}. 

So, let \(\Omega_1,\ldots,\Omega_t\) be the orbits of the action of \(G\) on \(\Omega\) and, for \(i\in\{1,\ldots t\}\), denote by \(K_i\) the kernel of the action of \(G\) on \(\Omega_i\); also, set \(m_i=\m{G/K_i}\). An application of \cite[Proposition~1]{CD} to the (transitive) action of \(G/K_i\) on the set \(\Omega_i\) yields that there exist disjoint subsets \(\Gamma_i,\;\Delta_i\) of \(\Omega_i\) such that 
\begin{enumeratei}
\item[(i)]  \(\pi_{\geq m_i}(G/K_i)=\pi_{\geq m_i}(G:G_{\Gamma_i}\cap G_{\Delta_i})\),
\item[(ii)]  \(|\pr_{m_i}|\leq 2|\pi(G:G_{\Gamma_i}\cap G_{\Delta_i})\cap\pr_{m_i}|,\)
\end{enumeratei}
for \(i\in\{1,\ldots, t\}\).

Next, define \[\Gamma=\bigcup_{i=1}^t\Gamma_i,\quad{\rm{and}}\quad\Delta=\bigcup_{i=1}^t\Delta_i,\]
so that \(\Gamma\) and \(\Delta\) are clearly disjoint subsets of \(\Omega\); it is also immediate to see that \(G_{\Gamma}=\bigcap_{i=1}^t G_{\Gamma_i}\) and \(G_{\Delta}=\bigcap_{i=1}^t G_{\Delta_i}\). Now, taking into account that \(\bigcap_{i=1}^t K_i=1\), we get  \(m={\rm{max}}\{m_1,\ldots,m_t\}\) and \[\pi_{\geq m}(G)\subseteq\bigcup_{i=1}^t \pi_{\geq m}(G/K_i)=\bigcup_{i=1}^t \pi_{\geq m}(G:G_{\Gamma_i}\cap G_{\Delta_i})\subseteq\pi_{\geq m}(G:\bigcap_{i=1}^t(G_{\Gamma_i}\cap G_{\Delta_i}))=\]\[=\pi_{\geq m}(G:(\bigcap_{i=1}^t G_{\Gamma_i})\cap(\bigcap_{i=1}^t G_{\Delta_i}))=\pi_{\geq m}(G:G_{\Gamma}\cap G_{\Delta}),\]
and (a) follows.  As for (b), observe that (we may assume \(m\neq 0\) and) there exists \(i\in\{1,\ldots,m\}\) such that \(G/K_i\) has a section isomorphic to \(A_m\); clearly \(m_i\) is then equal to \(m\), and our claim easily follows by (ii) observing that \(|G:G_{\Gamma_i}\cap G_{\Delta_i}|\) is a divisor of \(|G:G_\Gamma\cap G_\Delta|\).
\end{proof}

Note that, in the previous proposition, we can choose the two subsets \(\Gamma\) and \(\Delta\) of \(\Omega\) to be both non-empty unless \(G\) is the trivial group. In fact, assume that \(\Delta\) is empty; then, \(G\) being non-trivial, \(\Gamma\) is neither empty nor the whole \(\Omega\), and we can replace \(\Delta\) with \(\Omega\setminus\Gamma\).

\smallskip
Next, we sketch the proof of a result concerning almost-simple groups that can be essentially deduced from the proofs of Lemma~2.2 and Theorem~2.3 in \cite{HTV} (we refer the reader to that paper for the full details). 

Denote by $\mathcal{L}(p)$ the class of simple groups of Lie type in characteristic $p$; also, if $q$ is a prime power, denote by $ \ell_{n}(q) $ a primitive prime divisor of $q^n-1$ (that is a prime divisor of $q^n -1$, which does not divide $q^k -1$ for $1 \leq k < n$), if it exists.
By Zsigmondy's theorem (\cite[Theorem 6.2]{MW}), a primitive prime divisor of $q^n -1$ always exists unless
$n = 2$,  or $q=2$ and $n =6$.
%We use those notation in Table \ref{table:1} and  the following proposition.   

\begin{proposition}\label{HTV}
	Let $G$ be an almost-simple group with socle $S$.
	%Let $\pi_0 = \pi(G) - \pi(S)$ be the set of the prime divisors of $|G|$ that do not divide $|S|$.
	Then there exist two irreducible characters $\chi_1, \chi_2 \in \irr S$ such that the following conclusions hold. 
	\begin{enumeratei}
		\item If $S$ is either a sporadic simple group, $S \not\cong J_1$, or  $S \cong A_m$ for $m> 5$, or \(S\) is the Tits group, then $\pi(S) = \pi(G) =  \pi(\chi_1) \cup \pi(\chi_2)$;
		%\begin{description}
		%\item[(a1)]    
		if $S \cong J_1$, then $\pi(G) - (\pi(\chi_1) \cup \pi(\chi_2) ) = \{ 11 \}$.
		% (or $= \{ 19 \}$).
		%\end{description}  
		\item If $S \in \mathcal{L}(p)$, then $\pi(S) - (\pi(\chi_1) \cup \pi(\chi_2 )) \subseteq \{ p \}$.
		%; if  $S \cong {} ^2B_2(q)$ (Suzuki group), then $\pi(S) =  \pi(\chi_1) \cup \pi(\chi_2)$. 
		Moreover, 
		%\begin{description}
		%\item[(b1)] 
		$\pi(G) - \pi(S)$ is contained in both $\pi(|G:I_G(\chi_1)|)$ and  $\pi(|G:I_G(\chi_2)|)$.
		%\item[(b2)]   $\pi(\theta) \cap \pi(\psi) \neq \emptyset$ and $2 \in \pi(\theta) \cap \pi(\psi)$ if $p \neq 2$.
		%\end{description}
	\end{enumeratei} 
\end{proposition}
\begin{proof}
	If $S$ is either a sporadic simple group, or an alternating group, or the Tits group, then we have $\pi(S) = \pi(G)$ and claim (a) follows from \cite[Theorem~B(i), Theorem~C]{AB} and by \cite{atlas}  (actually, by \cite{BW} there exists  $\chi \in\irr{A_m}$ such that
	$\pi(A_m) = \pi(\chi)$ for $m \geq 15$). We remark that for \(A_5\) claim (a) does not hold, but this group can be treated as a simple group of Lie type (both in characteristic \(2\) and \(5\)) and, as such, it satisfies claim (b).
	
	Let now $S \in \mathcal{L}(p)$ (excluding the Tits group). The case \(S\cong\PSL{2^f}\) (for \(f\geq 2\)) is treated as Case 1 in \cite[proof of Theorem~2.3]{HTV}, so we will henceforth assume that \(S\) is not of this kind. 

Next, assume that at least one among the primitive prime divisors  $\ell_1$, $\ell_2$ indicated in Table \ref{table:1} (which summarizes \cite[Table~1]{HTV} and \cite[Lemma~2.3]{MT}) does not exist; then, it can be seen that \(S\) is isomorphic to a group of the kind $\PSL{q}$,  ${\rm{PSL}}_3{(q)}$, ${\rm{PSU}}_3(q)$ or ${\rm{PSp}}_4(q)$ for some specific values of \(q =p^f\), or $S$ belongs to a finite set of groups (see \cite[List~$\mathcal{C}$, page 3]{HTV}). Also, \(\pi(G)\) turns out to coincide with \(\pi(S)\) in this situation. The four infinite families above are discussed in \cite[Cases 1--3, proof of Lemma~2.2]{HTV}, whereas the remaining finite set of groups can be treated via \cite{atlas}. 
%If  $S$ is among   the following groups, then   looking at \cite{atlas} we get the results in part(b). 
%	\begin{center}
%		$\rm{PSL}_4(2)$,
%		$\rm{PSL}_6(2)$, $\rm{PSL}_7(2)$, $\rm{PSU}_4(2)$,
%		$\rm{PSU}_4(3)$, $\rm{PSU}_6(2)$,
%		$\rm{Sp}_4(2)'$, $\rm{Sp}_6(2)$, $\rm{Sp}_8(2)$,
%		$ \Omega_7(3)$, $\Omega^+_8(2)$,   $\Omega^-_8(2)$ $ 
%		^3D_4(2)$, $G_2(2)'$, $G_2(3)$, $G_2(4)$
%	\end{center}
%	If $S\cong \PSL{p^{f}}$ for some prime $p$, then part(b) of the Lemma would be deduced by \cite[Theorem A]{LW}.
%	Let $S\cong \rm {PSL}_3^ {\epsilon}(q)$ where $\epsilon=\pm 1$, $q=p^f$ for some prime $p$ and  $\pi(f)\subseteq \{2,3\}$.   Using  \cite{atlas} we may assume $p^f\geq 5$ and   looking at the character degrees $\chi_1(1)= (q - \epsilon)^2
%	(q + \epsilon)$  and $\chi_2(1)=q(q^
%	2 + \epsilon q + 1)$ (see \cite{SF}) we get that $\pi(G)=\pi(S)=\pi(\chi_1)\cup \pi(\chi_2)$. 
	
%	If  $S\cong \rm{Psp}_4(q)$, where $q=p^f>2$ and $\pi(f)\subseteq \{2,3\}$, then observing the character degrees $\chi_1(1)=\Phi_1^2\Phi_2^2$ and $\chi_2(1)=4\Phi_1\Phi_4^2$ (see \cite{SPS1, SPS2}), we see that $\pi(G)=\pi(S)=\pi(\chi_1)\cup\pi(\chi_2)$, as wanted.
	
As regards the cases that are left (for which both \(\ell_1\) and \(\ell_2\) exist) we introduce the following setup. Let $\mathscr{G}$ be a simply connected simple algebraic group defined over a field of order $q$ in characteristic $p$, and let $F$  be a suitable Frobenius map such that  $ S\cong \mathscr{G} ^{F}/\mathbf{Z}(\mathscr{G} ^{F})$. Suppose the pair  ($ \mathscr{G} ^{*}, F^{*} $)   is dual to   $ (\mathscr{G} , F) $. Setting  $L=\mathscr{G} ^{F}$   and    $ L^*=(\mathscr{G}^{*})^{F^{*}}$,  the  irreducible characters of $L$ are  partitioned into rational series $\mathscr{E}(L ,(s) )$ which are indexed by $(L^*)$-conjugacy
	classes $(s)$ of semisimple elements $s \in  L^*$, by Lusztig theory. Furthermore, if \({\rm{gcd}}(|s|, |{\bf Z}(L)|) = 1\), then every  $\chi \in \mathscr{E}(L ,(s) )$ is trivial at ${\bf Z}(L)$, thus $\chi\in \irr{S}$. Observe that $\chi(1)$  is divisible by $|L^*: \cent{L^*}s|_{p'}$. %Recall that  for  the remaining  simple groups of Lie type, both primitive prime divisors $\ell_i$ for $i\in \{1,2\}$ in Table \ref{table:1} exist. 
	Now, let $s_i\in L^*$  be a semisimple element of order $\ell_i$ for  $i\in \{1,2\}$: then \(\cent{L^*}{s_i}=T_i\) (see Table~1) is a  maximal torus of $L^*$.
	Note that  we have  $\gcd(\ell_i, |{\bf Z}(L)|)=1$ and also $\pi({\rm{gcd}}(|T_1|,|T_2|))\subseteq \pi(|L^*: T_1|_{p'})\cup \pi(|L^*: T_2|_{p'})$. Let $\chi_i\in \mathscr{E}(L, (s_i))$ for $i\in \{1,2\}$, so that \(\chi_i\in\irr S\)  and   $\chi_i(1)$ is divisible by $|L^*: T_i|_{p'}$ for \(i\in\{1,2\}\). Then $\pi(S)=\pi(\chi_1)\cup \pi(\chi_2)\cup \{p\}$. If $\pi(G)=\pi(S)$, then we get the desired conclusion. For the case when $\pi=\pi(G) - \pi(S)$ is nonempty, we refer to the discussion in \cite[Subcase~3b, proof of Theorem~2.3]{HTV}).	%Let $A$ be the Hall $\pi$-subgroup of $\rm{Out}(S)$. Note that  $A$ is  a cyclic and central subgroup of  ${\rm{Out}}(S)$, generated by a field automorphism. We claim that  $A$  intersect the inertia group for both $\chi_i$, $i \in \{ 1, 2\}$ in $G$, trivially.
	%It suffices to show that no field automorphism of $S$ of prime order can fix $\chi_i$
	%for
	%$i \in \{1,2\}$. Let $\gamma$ be a field automorphism of $S$ of prime order $r$. We can extend $\gamma$ to 
	%an automorphism of $L$
	%and  $L^*$
%	which we also denote  it by $\gamma$. Notice that $C_{L^*}(\gamma)$
%	is a finite group of Lie type of the same type as that of $L^*$
%	but defined over
%	a field of size $q^{1/r}$.
%	Now it is straightforward to check that both $\ell_i$,
%	$i \in \{1,2\}$, are
%	relatively prime to $|C_{L^*}(\gamma)|$. Hence the $L^*$-conjugacy classes $(s_i)$ of $s_i$
%	in $L^*$ are
%	not  $\gamma$-invariant for $i\in \{1,2\}$. Then $\gamma(s_i)$ and $s_i$ are not
%	$L^*$-conjugate for $i \in \{ 1, 2\}$, and thus $\chi_i  \in \mathscr{E}(L,(s_i))$ for  $i \in \{ 1, 2\}$, are not $\gamma$-invariant as claimed.  Therfore $\pi\subseteq \pi(|G:I_G(\chi_i)|)$ for each $i\in \{1,2\}$ and we obtain $\pi(G)=\pi(\chi_1^G)\cup \pi(\chi_2^G)\cup \{p\}$. 
\end{proof}

\begin{table}[h!]\label{table1}
	\begin{center}	
		\caption{Two tori for  groups of Lie type in characteristic $p$}
		\scalebox{0.78}{
			\begin{tabular}{| c | c | c | c | c |} 
				\hline
				$ G = G(q) $, $q=p^f$  & $ |T_{1}| $ & $ |T_{2}| $ & $ \ell _{1} $ & $ \ell _{2} $\\
				\hline		
				$A_{n}(q) $ & $ (q^{n+1}-1)/(q-1) $ & $ q^{n}-1 $ & $ \ell_{n+1}(q) $ & $ \ell_{n}(q) $\\ [0.5ex]  
				${}^2A_{n}(q), n\equiv 0(4) $ & $ (q^{n+1}+1)/(q+1)  $ & $ q^{n}-1 $ & $ \ell_{2n+2}(q) $ &  $ \ell_{n}(q) $ \\
				[0.5ex] 
				${}^2A_{n}(q), n\equiv 1(4) $ & $ (q^{n+1}-1)/(q+1) $ & $ q^{n}+1 $ & $ \ell_{(n+1)/2}(q) $ &  $ \ell_{2n}(q) $\\
				[0.5ex] 
				${}^2A_{n}(q), n\equiv 2(4)  $ & $  (q^{n+1}+1)/(q+1) $ & $ q^{n}-1 $ & $ \ell_{2n+2}(q) $ &  $ \ell_{n/2}(q) $ \\
				[0.5ex] 
				${}^2A_{n}(q), n\equiv 3(4)  $ & $ (q^{n+1}-1)/(q+1) $ & $ q^{n}+1 $ & $ \ell_{n+1}(q) $ & $ \ell_{2n}(q)  $ \\ 
				[0.5ex] 
				$B_{n}(q), C_{n}(q), n \geq 3 \ odd  $ & $ q^{n}+1 $ & $ q^{n}-1 $  &  $ \ell_{2n}(q)  $ & $ \ell_{n}(q)   $ \\
				[0.5ex] 
				$B_{n}(q), C_{n}(q),  n \geq 2 \ even  $ &  $ q^{n}+1  $ &  $ (q^{n-1}+1)(q+1) $ &  $ \ell_{2n}(q)  $ & $ \ell_{2n-2}(q) $\\
				[0.5ex] 
				$D_{n}(q), n \geq 5 \  odd   $ & $ (q^{n-1}+1)(q+1) $ & $ q^{n}-1  $ & $ \ell_{2n-2}(q) $ & $ \ell_{n}(q) $ \\
				[0.5ex] 
				$D_{n}(q), n \geq 4 \ even  $ & $ (q^{n-1}+1)(q+1) $ & $ (q^{n-1}-1)(q-1) $ & $  \ell_{2n-2}(q) $ & $ \ell_{n-1}(q) $ \\
				[0.5ex] 
				$^{2}D_{n}(q) $ & $  q^{n}+1 $ & $ (q^{n-1}+1)(q-1) $ & $  \ell_{2n}(q) $ & $  \ell_{2n-2}(q)  $\\
				[0.5ex] 
				$ ^2B_2(q)$ & $q\pm\sqrt{2q}+1$  &  $q-1$ & $\ell_{4}(q)$ & $\ell_{1}(q)$\\
				[0.5ex] 
				$^2G_2(q)$ & $q\pm \sqrt{3q}+1$ & $q-1$ & $\ell_{6}(q)$ & $\ell_{1}(q)$\\
				[0.5ex] 
				$^2F_4(q)$ & $q^2+q+1\pm(\sqrt{3q}+ \sqrt{2q})$ & $q^2-q+1$ & $\ell_{12}(q)$  & $\ell_{6}(q)$\\	
				[0.5ex] 	
				$^3D_4(q)$ & $ q^4-q^2+1$ & $(q^3+1)(q+1)$ & $\ell_{12}(q)$ & $\ell_{6}(q)$\\
				[0.5ex] 
				$G_2(q)$ & $\Phi_6(q)$ &‌$\Phi_3(q)$ & $\ell_{6}(q)$ & $\ell_{3}(q)$\\	
				[0.5ex] 
				$F_4(q)$ & $\Phi_{12}(q)$ &‌$\Phi_8(q)$ & $\ell_{12}(q)$ & $\ell_{8}(q)$\\
				[0.5ex] 
				$E_8(q)$ & $\Phi_{30}(q)$ &‌$\Phi_{24}(q)$ & $\ell_{30}(q)$ & $\ell_{24}(q)$\\
				[0.5ex] 
				$E_6(q)$ & $\Phi_9(q)$ &‌$\Phi_{12}(q)\Phi_3(q)$ & $\ell_{9}(q)$ & $\ell_{12}(q)$\\
				[0.5ex] 
				$^2E_6(q)$ & $\Phi_{18}(q)$ &‌$\Phi_{12}(q)\Phi_6(q)$ & $\ell_{18}(q)$ & $\ell_{12}(q)$\\
				[0.5ex] 
				$E_7(q)$ & $\Phi_{18}(q)\Phi_2(q)$ &‌$\Phi_{14}(q)\Phi_2(q)$ & $\ell_{18}(q)$ & $\ell_{14}(q)$\\										
				[1ex] 
				\hline
			\end{tabular}
		}
		\label{table:1}
	\end{center}
\end{table}

\smallskip	
As the last preliminary result we recall the statement of Theorem~A in \cite{HTV}, that proves the strengthened \(\rho\)-\(\sigma\) conjecture for almost-simple groups.
	
\begin{theorem}\label{TV}
Let $G$ be an almost-simple group. Then $|\pi(G)|\leq 2\sigma(G)$ unless $G\cong\PSL{2^f}$ with $f\geq2$ and $|\pi(2^f-1)| = |\pi(2^f + 1)|$. For the exceptions, we have $|\pi(G)|\leq 2\sigma(G)+1$.
\end{theorem}

\section{The $\rho$-$\sigma$ conjecture for groups with trivial Fitting subgroup}

In this section we present a proof of Theorem~A, that was stated in the Introduction. The result is obtained by combining Theorem~\ref{lessthan5} and Theorem~\ref{morethan5}, which provide the bounds in the case \(\sigma(G)\leq 5\) and \(\sigma(G)\geq 6\) respectively.

Recall that a \emph{component} of a group \(G\) is a non-trivial subnormal subgroup of \(G\) which is \emph{quasi-simple} (i.e., a perfect group whose factor group over its centre is simple). Denoting by \(\E G\) the subgroup generated by all  the components of \(G\), the \emph{generalized Fitting subgroup} \(\fitg G\) of \(G\) is then defined as the product \(\fit G\E G\). In the proof of Theorem~\ref{key} it will be useful to take into account that, if \(\fit G\) is trivial, then \(\fitg G\) is the product of all the minimal normal subgroups of \(G\) (see for instance \cite[6.5.5(b)]{KS}), and also that distinct components of \(G\) centralize each other (\cite[6.5.3]{KS}).

\begin{theorem} \label{key}
Let  $G$ be a group with  $\fit{G}=1$, and assume that  $G$ does not have any simple  characteristic subgroup. Then there exist $\psi_1, \psi_2\in \irr{\fitg{G}}$ (not necessarily distinct) 
such that,   %$\pi(\psi_1)\cap \pi(\psi_2)\not= \emptyset$,  and also 
setting $m=\m{G/\fitg{G}}$, the following conclusions hold.
\begin{enumeratei}
\item $\pi_{\geq m}(G)\subseteq \bigcup^2_{i=1} \big( \pi(\psi_i)\cup \pi(G:I_G(\psi_i))\big)$;
\item $|\pr_m|\leq 2|\pi(G:I_G(\psi_i))\cap\pr_m|$  for $i=1,2$.
\end{enumeratei}
\end{theorem}
\begin{proof}
  We argue by induction on $|G|$. Let $M$ be a minimal characteristic subgroup of $G$, and define $N=\cent GM$ (note that \(N\) is a characteristic subgroup of \(G\) as well).
  Then $\fit{G/N}=1$ and $\fitg{G/N}=MN/N$, so the factor group of \(G/N\) over its generalized Fitting subgroup is isomorphic to \(G/MN\); also, it is easy to see that \(G/N\) does not contain any simple characteristic subgroup.
  Assume first that $N\not =1$, and set $m_1=\m{G/MN}$. By induction, there exist $\b{\lambda}_1 , \b{\lambda}_2\in \irr{MN/N}$ such that 
	
	$$\pi_{\geq m_1}(G/N)\subseteq \bigcup^2_{i=1}\left( \pi(\b{\lambda}_i)\cup \pi(G/N:I_{G/N}(\b{\lambda}_i)) \right), $$ 
	%and $\pi(\b{\lambda_1})\cap \pi(\b{\lambda_2})\not =\emptyset$.
        and also, $|\pr_{m_1}|\leq 2|\pi(G/N: I_{G/N}(\b{\lambda_i}))\cap\pr_{m_1}|$ for $i\in\{1,2\}$. Regarding \(\overline{\lambda}_i\) as a character of \(MN\) by inflation, 
let $\lambda_i\in \irr{M}$ be such that \(\lambda_i\times 1_N=\overline{\lambda}_i\). %the lift of $\b{\lambda_i}$ to $MN = M \times N$. 
        Then   %$\lambda_i = \b{\lambda_i} \times 1_N$ and , 
        clearly $I_G(\lambda_i)/N = I_{G/N}(\b{\lambda}_i)$.
        
	Observe now that $\fit{N}=1$. Moreover, taking into account the paragraph preceding this theorem, we have $\fitg{G}=M\times \fitg{N}$. Note also that, obviously, \(N\) does not contain any simple characteristic subgroup. Let $m_2= \m{N/\fitg{N}}$. By induction,  there exist $\mu_1$ and $\mu_2\in \irr{ \fitg{N}}$, such that 	
	$$\pi_{\geq m_2}(N)\subseteq \bigcup_{i=1}^2 \big(\pi(\mu_i)\cup \pi(N:I_N(\mu_i)) \big)$$
		%and $\pi(\mu_1)\cap \pi(\mu_2)\not =\emptyset$.
		 and $|\pr_{m_2}|\leq 2|\pi(N:I_N(\mu_i))\cap\pr_{m_2}|$, for $i\in\{1,2\}$.
		
                Take now $\psi_1=\lambda_1\times \mu_1$ and \(\psi_2=\lambda_2\times\mu_2\) in $\irr{\fitg{G}}$.
                %Clearly, $\pi(\psi_1)\cap \pi(\psi_2)\not =\emptyset$. 
                Since $G/\fitg{G}$ is an extension of $MN/\fitg{G}\cong N/\fitg{N}$,
                by  $G/MN$, we have $m={\rm{max}}\{m_1,m_2\}$ by  Lemma~\ref{lm}.
                As $I_G(\psi_i)=I_G(\lambda_i)\cap I_G(\mu_i)$,  we conclude that  $\pi(G:I_G(\lambda_i))\cup \pi(N:I_N(\mu_i))\subseteq \pi(G:I_G(\psi_i))$, and so (a) holds. Moreover, since \(\pr_m\) coincides either with $\pr_{m_1}$ or with $ \pr_{m_2}$, claim (b) follows as well. 
	
                Hence we may assume $N=1$, so $M=\fitg{G}$ is the socle of \(G\).
                Now, $M$ is the direct product of a set $\Omega$ of subgroups that are all isomorphic to a suitable non-abelian simple group \(S\), and that are permuted by conjugation by $G$. Denoting by $K$ the kernel of this action, we remark that  the permutation group \(G/K\) on \(\Omega\) is non-trivial because \(M\) is not a simple group.
Note also that \(K\) is a characteristic subgroup of \(G\), and that \(\aut G\) transitively permutes the set \(\{\cent K S\mid S\in\Omega\}\); as a consequence, the factor groups \(K/\cent K S\) are pairwise isomorphic when \(S\) ranges over \(\Omega\).

                Noting that  $K/M$ is solvable, since it is a subgroup of a direct product of copies of $\out S$, we see that  $\m{G/K}=\m{G/M} = m$.
                By Proposition~\ref{permutation}, there exist two disjoint subsets $\Gamma, \Delta \subseteq \Omega$ such that
	
	$$\pi_{\geq m}(G/K) = \pi_{\geq m}(G: G_{\Gamma}\cap G_{\Delta})$$
and 

$$|\pr_m|\leq 2|\pi(G: G_{\Gamma}\cap G_{\Delta})\cap\pr_m|;$$
as observed in the paragraph following Proposition~\ref{permutation}, we can assume that both $\Gamma$ and $\Delta$ are non-empty.

Now, take \(S\in\Omega\) and define \(C=\cent K S\). An application of Proposition~\ref{HTV} to the almost-simple group \(K/C\) (of socle \(SC/C\simeq S\)) yields the existence of \(\chi_1,\chi_2,\xi\) in \(\irr S\) such that \(\pi(\chi_1)\cup\pi(\chi_2)\) contains all the primes in \(\pi(S)\), except possibly a single prime \(p\) that is recovered as an element of \(\pi(\xi)\), and both \(|K:I_K(\chi_1)|\), \(|K:I_K(\chi_2)|\) contain all the primes in \(\pi(K/C)-\pi(S)\). Writing \(\Omega=\{S_1,\ldots S_k\}\), we choose \(\chi_1^{(i)},\chi_2^{(i)},\xi^{(i)}\in\irr{S_i}\) in this way; since the groups \(K/\cent K{S_i}\) are pairwise isomorphic, we can assume \(\chi_1^{(i)}(1)=\chi_1^{(j)}(1)\) and \(\chi_2^{(i)}(1)=\chi_2^{(j)}(1)\) for all \(i,j\in\{1,\ldots,k\}\).

Next, we define \(\psi_1\in\irr M\) as follows. For \(S_i\) belonging to \(\Gamma\) we take the character \(\chi_1^{(i)}\), for \(S_j\) in \(\Delta\) we take \(\chi_2^{(j)}\), and for \(S_{t}\) in \(\Omega-(\Gamma\cup\Delta)\) we take the trivial character; then, we set \(\psi_1\) to be the direct product of all these characters. Similarly we define \(\psi_2\in\irr M\), just replacing \(\chi_2^{(j)}\) with \(\xi^{(j)}\) for all \(S_j\in\Delta\), so we clearly have \(\pi(M)\subseteq\pi(\psi_1)\cup\pi(\psi_2)\). Moreover, as \(I_K(\psi_1)\) and \(I_K(\psi_2)\) both lie in \(\bigcap_{i=1}^k I_K(\chi_1^{(i)})\), we see that \(\pi(G:I_G(\psi_1))\) and \(\pi(G:I_G(\psi_2))\) (which clearly contain \(\pi(K:I_K(\psi_1))\) and \(\pi(K:I_G(\psi_2))\), respectively) both contain \(\bigcup_{i=1}^k\pi(K/\cent K {S_i})-\pi(S)\). In view of the fact that \(\bigcap_{i=1}^k\cent K{S_i}\) is trivial, so that \(K\) embeds in the direct product of the groups \(K/\cent K{S_i}\), we conclude that \(\bigcup^2_{i=1} \pi(\psi_i)\cup \pi(G:I_G(\psi_i))\) contains \((\pi(K)-\pi(M))\cup\pi(M)=\pi(K)\).

Moreover, \emph{assuming that the degrees of \(\chi_1\),\(\chi_2\) are distinct} (and swapping the role of \(\chi_1\), \(\chi_2\) if \(\chi_1(1)=\xi(1)\)), it is easy to see that both \(I_G(\psi_1)\) and \(I_G(\psi_2)\) lie in \(G_{\Gamma}\cap G_{\Delta}\), therefore \(\pi(G:I_G(\psi_1))\) and \(\pi(G:I_G(\psi_2))\) contain \(\pi_{\geq m}(G/K)\). Thus claim (a) is proved, and also claim~(b) follows immediately.

If \(\chi_1(1)=\chi_2(1)\neq\xi(1)\), then one character (namely \(\psi_2\)) is enough to obtain the desired conclusions. If, finally, we have \(\chi_1(1)=\chi_2(1)=\xi(1)\), then the argument needs a small adjustment: namey, in the definition of \(\psi_1\), for the groups \(S_j\) belonging to \(\Delta\) we choose any non-linear irreducible character whose degree is different from \(\chi_1(1)\) in place of \(\chi_2^{(j)}\), and \(\psi_1\) so modified is enough to obtain the desired conclusions.
\end{proof}
 
 \begin{corollary}\label{key2}
   If $G$ is a group such that $\fit{G}=1$  and  $G$ does not have any simple  characteristic subgroup, then $|\pi(G)|\leq 2\sigma(G)$. 
 \end{corollary}
\begin{proof}
%	Note that $\fit{G/\fit{G}}$ is trivial. So we can apply Theorem \ref{key} on $G/\fit{G}$.  Note that $|V(\fit{G})|=\sigma(\fit{G})$. 
  Let $\psi_1, \psi_2$ be the characters in $\irr{\fitg{G}}$ provided by Theorem~\ref{key} and, for \(i\in\{1,2\}\), take $\chi_i\in \irr{G|\psi_i}$. By Clifford Theory we have  $\pi(\psi_i)\cup \pi(G:I_G(\psi_i))\subseteq \pi(\chi_i)$, whence, for some $i\in \{1,2\}$, we get $2|\pi(\chi_i)|\geq |\pi(G)|$ by Theorem~\ref{key}.
	\end{proof}

We remark that the argument of Theorem~\ref{key} can be easily adapted to prove the following statement.

\smallskip
\noindent{\sl{Let  $G$ be a group with $\fit{G}=1$, and assume that  $G$ does not have  any simple  characteristic subgroup. Assume further that \(|G|\) is not divisible by \(3\). Then there exist $\psi\in \irr{\fitg{G}}$ such that $\pi(G)=\pi(\psi)\cup \pi(G:I_G(\psi))$.}}

\smallskip 
To this end, one should take into account that the assumption of \(|G|\) not being a multiple of \(3\) yields \(\m G=0\); moreover, the same assumption yields that every simple subnormal subgroup \(S\) of \(G\) is a Suzuki group, and so it is possible to find \(\chi_1\), \(\chi_2\) in \(\irr S\) such that \(\pi(\chi_1)\cup\pi(\chi_2)=\pi(S)\) (see \cite[Page 182]{HB}).

 As a consequence it is immediate to see that, under the same hypotheses, we have \(|\pi(G)|=\sigma(G)\). We omit the full proofs of these facts because they amount to a straightforward modification of what already seen in Theorem~\ref{key} and Corollary~\ref{key2}, and also because these results are not essential for the rest of this paper.

\medskip
We are now in a position to prove the two theorems that together constitute Theorem~A.

\begin{theorem}\label{lessthan5}
Let $G$ be a group with trivial Fitting subgroup. If \(\sigma(G)\leq 5\), then we have \(|\pi(G)|\leq2\sigma(G)+1\).
	%\begin{enumeratei}
		%\item If  $\sigma(G)\leq 5$, then $|\pi(G)|\leq 2\sigma(G)+1$.		
    	%\item	If $\sigma(G)\geq 6$, then $|\pi(G)|\leq 3\sigma(G)-4$.
    %\end{enumeratei}
\end{theorem}
\begin{proof}

Let \(G\) be a counterexample of minimal order to the statement. We start by observing that, by Corollary~\ref{key2}, the set of simple normal subgroups of \(G\) is non-empty. So, let us consider any simple normal subgroup \(S\) of \(G\), and set \(C=\cent G S\); note that \(C\) cannot be trivial, as otherwise \(G\) would be an almost-simple group and, in view of Theorem~\ref{TV}, it would not be a counterexample. Moreover, \(C\) clearly satisfies our assumptions and, being strictly smaller than \(G\), it also satisfies the conclusion \(|\pi(C)|\leq2\sigma(C)+1\).

Let us fix the following convention. If the simple group \(S\) is in \(\mathcal{L}(p)\), then we call \(p\) the \emph{associated prime} of \(S\) (in the cases when \(S\) has multiple characteristic, we choose \(p\) to be the smallest among them) and, if \(S\cong J_1\), then we set \(11\) to be the associated prime of \(S\); for any other isomorphism type of \(S\), we say that \(S\) has no associated prime. Now we apply Proposition~\ref{HTV} to the almost-simple group \(G/C\), whose socle is isomorphic to \(S\): there exist \(\chi_1\) and \(\chi_2\) in \(\irr {S}\) such that $\pi(S)=\pi(\chi_1) \cup \pi(\chi_2)$ unless \(S\) has an associated prime \(p\), in which case we have  $\pi(S)-(\pi(\chi_1) \cup \pi(\chi_2))\subseteq\{p\}$. Moreover, defining \(\pi_0=\pi(G/C)-\pi(S)\), we have $\pi_0 \subseteq \pi(|G:I_G(\chi_i)|)$ for \(i\in\{1,2\}\). If we set $\pi_1 = \pi(\chi_1) \cup \pi_0$, $\pi_2 = \pi(\chi_2) \cup \pi_0$, and $\pi_3$ to be either $\{p\}$ or $\emptyset$ according on whether \(S\) has an associated prime \(p\) or not, we can write  
%By Clifford theory, for $i$ in $\{1,2,3\}$, if  $\gamma_i$ lies in  $\irr {G|\chi_i}$, then $\pi_i$ is contained in $\pi(\gamma_i)$; hence 
$\pi(G) = \pi(C) \cup \bigcup_{i=1}^3 (\pi_i - \pi(C))$, and therefore
\begin{equation} \label{eq1}
  |\pi(G)| \leq |\pi(C)| + \sum_{i=1}^3 |\pi_i - \pi(C)|.%\leq|\pi(C)|+3\max_{i\in\{1,2,3\}}\{|\pi_i - \pi(C)|\}.
\end{equation}
%Hence, $\pi(G/C) \subseteq  \pi(\gamma_1)\cup \pi(\gamma_2)\cup \pi(\gamma_3)$.
(Note that \(|\pi_3-\pi(C)|\leq 1\).)
If $S$ has an associated prime $p$, then we take \(\chi_3\in\irr{S}\) whose degree is divisible by \(p\); otherwise we just take $\chi_3= 1_S$. 
Let \(\delta\in\irr C\) such that \(|\pi(\delta)|=\sigma(C)\), for \(i\in\{1,2,3\}\) we can consider irreducible characters of $G$ lying above  $\delta \times \chi_i$, and we deduce that
\begin{equation}\label{eq2}
\sigma(G) \geq  \sigma(C) + \max_{i\in\{1,2,3\}}\{|\pi_i - \pi(C)|\}.
\end{equation}

\medskip
As the next step, we will prove a number of claims.  

\smallskip
\noindent{\bf{(i)}} {\sl{For \(i\neq j\in\{1,2\}\), \(\pi(\chi_i)-\pi(C)\) is not contained in \(\pi_j\cup\pi_3\), and \(\pi_3-\pi(C)\) is not contained in \(\pi_1\cup\pi_2\); in particular, \(\pi(\chi_1)-\pi(C)\), \(\pi(\chi_2)-\pi(C)\), \(\pi_3-\pi(C)\) are non-empty, and \(\pi(S)\) contains at least two primes (one in \(\pi(\chi_1)-\pi(\chi_2)\) and the other in \(\pi(\chi_2)-\pi(\chi_1)\)) that are not in \(\pi(C)\cup\pi_0\cup\pi_3\).}}

For a proof by contradiction, assume that \(\pi(\chi_1)-\pi(C)\) is contained in \(\pi_2\cup\pi_3\); then, taking into account that \(\pi_0\) lies in \(\pi_2\), we get $\pi(G) = \pi(C) \cup \bigcup_{i=2}^3 (\pi_i - \pi(C))$, and Inequality~(1) gives \[|\pi(G)|\leq|\pi(C)|+2\max_{i\in\{2,3\}}\{|\pi_i - \pi(C)|\}.\] Recalling that \(|\pi(C)|\leq2\sigma(C)+1\), we obtain \[|\pi(G)|\leq 2(\sigma(C)+\max_{i\in\{2,3\}}\{|\pi_i - \pi(C)|\})+1\] and therefore, in view of Inequality (2), we get the contradiction \(|\pi(G)|\leq2\sigma(G)+1\). The same argument of course works for the other inclusions as well, and the remaining parts of the claim can be easily deduced.

\smallskip
The fact that \(\pi_3-\pi(C)\) is non-empty yields, on one hand, that \(\sigma(G)\) is strictly larger than \(\sigma(C)\) (taking into account Inequality (2)). On the other hand, \(S\) must have an associated prime which does not lie in \(\pi(C)\): 

\smallskip
\noindent{\bf{(ii)}} {\sl{\(S\) is either of Lie type in odd characteristic, or \(S\cong J_1\). In particular, \(|S|\) is divisible by \(3\).}} 

Here, the only thing to observe is that \(S\) cannot lie in \(\mathcal{L}(2)\) because \(C\) has even order, as \(C\neq 1\) and \(\fit C=1\).

\smallskip
%Clearly {\bf{(ii)}} holds for any simple normal subgroup of \(G\), and 
Since two (distinct) simple normal subgroups of \(G\) centralize each other, we also deduce the following: 

\smallskip
\noindent{\bf{(iii)}} {\sl{The associated primes of the simple normal subgroups of \(G\) are pairwise distinct.}} 

\smallskip
Some other remarks concerning the centralizer \(C\) of a simple normal subgroup $S$ of $G$:

\smallskip
\noindent{\bf{(iv)}} {\sl{\(|\pi(C)|=2\sigma(C)+1\). In addition, \(C\) has a simple characteristic subgroup and \(|C|\) is divisible by \(3\); furthermore, \(C\) is not an almost-simple group.}} 

In fact, we know that \(|\pi(C)|\leq 2\sigma(C)+1\), so, if the claim is false, then we have \(|\pi(C)|\leq 2\sigma(C)\). Now, using again inequalities (1) and (2) as in {\bf{(i)}} and recalling that \(|\pi_3-\pi(C)|= 1\), we get \(|\pi(G)|\leq2\sigma(G)+1\), which is not the case. 
Given that, \(C\) must have a simple characteristic subgroup by Corollary~\ref{key2} and, as this subgroup is a simple normal subgroup of \(G\), by {\bf{(ii)}} its order is divisible by \(3\). Finally, if \(C\) is an almost-simple group, its socle (which is a simple normal subgroup of \(G\)) in this situation must be then isomorphic to a projective special linear group in characteristic \(2\) by Theorem~\ref{TV}, and this is against claim {\bf{(ii)}}.

\smallskip
Since we have seen that \(3\) divides \(|C|\), claim {\bf{(i)}} (namely, \(\pi_3-\pi(C)\neq\emptyset\)) yields also:

\noindent{\bf{(v)}} {\sl{\(S\not\in\mathcal{L}(3)\)}}.

\smallskip 
Next,

\smallskip
\noindent{\bf{(vi)}} {\sl{Both \(\sigma(S)\) and \(\sigma(C)\) lie in \(\{3,4\}\).}} 

To the end of showing \(\sigma(S)\geq 3\), by Lemma~2.4 in \cite{HTV} we only have to rule out the case \(S\cong\PSL q\) (where \(q\) is a \(p\)-power with \(p>3\)) and \(|\pi(q\pm 1)|\leq 2\); but in this situation the degree of (say) \(\chi_1\), as defined in the second paragraph of this proof, is divisible only by \(2\) and \(3\), so \(\pi(\chi_1)-\pi(C)\) is empty, against {\bf{(i)}}. As for \(\sigma(C)\leq 4\), this follows at once by the fact that, as observed, we have \(\sigma(C)<\sigma(G)\); the remaining conclusions also follow easily, taking into account that (by {\bf{(iv)}}) \(C\) contains a simple normal subgroup of \(G\).  

\smallskip
Consider now the set \(\mathcal{K}=\{S_1,S_2,\ldots,S_k\}\) of all the simple normal subgroups of \(G\). As our last preliminary claim, we show the following:

\smallskip
\noindent{\bf{(vii)}} {\sl{\(G=(S_1\times S_2\times\cdots\times S_k)\times U\), where \(|\pi(U)|\leq 2\sigma(U)\).}} 

Let \(S\in\mathcal{K}\). Recalling that (by {\bf{(vi)}}) we have \(\sigma(C)\in\{3,4\}\), assume first \(\sigma(C)=3\) (which implies \(|\pi(C)|= 2\cdot 3+1=7\)) and \(\sigma(G)=5\): then Inequality (2) yields \(|\pi_i-\pi(C)|\leq 2\) for \(i\in\{1,2,3\}\). Now, \(\pi_1-\pi(C)\) cannot contain any element \(u\) of \(\pi_0\), as otherwise \(u\) would lie in \(\pi_2-\pi(C)\) as well, and the set \(X=(\pi_1-\pi(C))\cup(\pi_2-\pi(C))\) would contain at most three elements. As a consequence of the fact that \(\pi(G)\) is covered by \(\pi(C)\cup X\cup(\pi_3-\pi(C))\) we would then obtain \(|\pi(G)|\leq 7+3+1=11=2\sigma(G)+1\), against our assumption. We conclude that \(\pi(G)=\pi(SC)\) in this case. 

But the same holds also whenever \(\sigma(G)-\sigma(C)=1\) (which covers all the remaining possibilities): in fact, in the latter situation, \(\pi_i-\pi(C)\) is a singleton and (by {\bf{(i)}}) it must be contained in \(\pi(\chi_i)\) for \(i\in\{1,2\}\). 

Now, if \(SC\) is strictly contained in \(G\), our minimality assumption leads again to the contradiction \(|\pi(G)|=|\pi(SC)|\leq2\sigma(SC)+1\leq2\sigma(G)+1\). Thus we have \(G=SC\) and, since this holds for all the elements in \(\mathcal{K}\), we easily deduce that \(G=(S_1\times S_2\times\cdots\times S_k)\times U\) where \(U=\bigcap_{i=1}^k\cent G{S_i}\). Finally, observe that either \(U\) is trivial, or it is a group with trivial Fitting subgroup and no simple characteristic subgroups. Therefore, by Corollary~\ref{key2}, the desired conclusion follows.

\medskip
We observed in {\bf{(iv)}} that, for \(S\in\mathcal{K}\), there is an element of \(\mathcal{K}\) lying in \(\cent G S\), thus \(|\mathcal{K}|\neq 1\). We work next to exclude also \(|\mathcal{K}|>1\), thus obtaining a contradiction.

Note that, if \(S_i\) is in \(\mathcal{K}\), then the two characters \(\chi^{(i)}_1\), \(\chi^{(i)}_2\in\irr{S_i}\) given by Proposition~\ref{HTV} ``cover" together all the prime divisors of \(|S_i|\) except possibly the associated prime \(p_i\) of \(S_i\), as explained in the second paragraph of this proof (recall that, by {\bf{(ii)}} and {\bf{(v)}}, we have \(p_i\not\in\{2,3\}\)); in particular, taking into account that \(\pi(\chi^{(i)}_1)-\pi(\cent G{S_i})\) and \(\pi(\chi^{(i)}_2)-\pi(\cent G{S_i})\) are non-empty by {\bf{(i)}}, we see that one among \(\chi^{(i)}_1\), \(\chi^{(i)}_2\) has a degree divisible by \(2r_i\), and another (possibly the same) has a degree divisible by \(3s_i\), where \(r_i\) and \(s_i\) are (possibly equal) primes not dividing \(|\cent G{S_i}|\).
Note that hence, clearly, $\{r_i, s_i\} \cap \{r_j, s_j\} = \emptyset$ for $i \neq j$.
Given that, assume \(|\mathcal{K}|\geq 4\) and take four distinct elements \(S_i\), \(i\in\{1,\ldots,4\}\), in \(\mathcal{K}\): considering irreducible characters \(\xi_i\) of \(S_i\) such that \(\{2,r_1\}\subseteq\pi(\xi_1)\), \(\{3,s_2\}\subseteq\pi(\xi_2)\), \(p_3\in\pi(\xi_3)\) and \(p_4\in\pi(\xi_4)\), we see that \(|\pi(\xi_1\times\cdots\times\xi_4)|\geq 6\), against the assumption \(\sigma(G)\leq 5\). Our conclusion so far is that \(|\mathcal{K}|\leq 3\).

\smallskip
Assume then \(\mathcal{K}=\{S_1,S_2,S_3\}\), and observe that in this case we have \(\sigma(S_i)=3\) for every \(i\in\{1,2,3\}\): in fact, if (say) \(S_1\) has an irreducible character \(\xi_1\) with \(|\pi(\xi_1)|\geq 4\) then, considering \(\xi_i\in\irr{S_i}\) with \(p_i\in\pi(\xi_i)\) for \(i\in\{2,3\}\), by {\bf (iii)} we get \(|\pi(\xi_1\times\xi_2\times\xi_3)|\geq 6\), again a contradiction. 
Furthermore, assume that an element \(S\) of \(\mathcal{K}\) is not a \(2\)-dimensional projective special linear group; then, we claim that there exists an irreducible character \(\xi\) of \(S\) whose degree is divisible by \(p\) and by a prime \(t\not\in\{2,3\}\). Once this will be established, setting \(S_3=S\), we can choose \(\xi_3=\xi\in\irr{S_3}\), and \(\xi_1\in\irr{S_1}\), \(\xi_2\in\irr{S_2}\) as in the paragraph above, getting the contradiction \(|\pi(\xi_1\times \xi_2\times\xi_3)|\geq 6\). So, we work to prove this claim.

According to Theorem~6.3 of \cite{L} and recalling that, by {\bf{(i)}}, \(\pi(S)\) contains at least two primes \(r\), \(s\) other than \(2\), \(3\) and \(p\), the only cases that must be treated are when either \(S\cong{\rm{PSL}}_3(p^f)\) or \(S\cong{\rm{PSU}}_3(p^f)\) for some positive integer \(f\). In the former case, looking at Theorem~3.2 of \cite{W}, our claim is true unless (setting \(q=p^f\)) we have \(\pi((q+1)(q^2+q+1))\subseteq\{2,3\}\); this forces the odd number \(q^2+q+1\) to be a power of \(3\), and the primes \(r\), \(s\) to lie in \(\pi(q-1)\); but \(S\) has an irreducible character \(\psi\) of degree \((q-1)(q^2+q+1)\), which implies \(\pi(\psi)\supseteq\{2,3,r,s\}\) and contradicts \(\sigma(S)=3\). The case \(S\cong{\rm{PSU}}_3(p^f)\) can be treated similarly, referring to \cite[Theorem~3.4]{W}.
This establishes the claim in the above paragraph.

In order to rule out the case \(\mathcal{K}=\{S_1,S_2,S_3\}\), it remains then to treat the situation when \(S_i\) is a \(2\)-dimensional projective special linear group (in characteristic larger than \(3\)) for every \(i\in\{1,2,3\}\). The set of character degrees of \(\PSL q\), where \(q > 5\) is a prime power (note that $\PSL 5 \cong \PSL 4$ does not show up here by {\bf (ii)}), is well known (see Example~\ref{ex}). In our situation, taking into account {\bf{(i)}} and the fact that \(\sigma(S_i)=3\) for all \(i\in\{1,2,3\}\), each \(S_i\) has two irreducible characters \(\chi_1^{(i)}\), \(\chi_2^{(i)}\) such that \(\pi(\chi_1^{(i)})\supseteq\{2,3,r_i\}\) and \(\pi(\chi_2^{(i)})\supseteq\{2,s_i\}\cup T_i\), where \(r_i, s_i\) are primes not lying in \(|\cent G{S_i}|\) and \(T_i\) is either empty or a singleton \(\{t_i\}\). But every \(T_i\) must be in fact empty: if (say) \(t_1\in T_1\), then \(\chi_2^{(1)}\times\chi_1^{(2)}\in\irr{S_1\times S_2}\) has a degree divisible by \(2,3,s_1,t_1,r_2\), and choosing an irreducible character of \(S_3\) whose degree is divisible by \(p_3\) we  get the contradiction \(\sigma(S_1\times S_2\times S_3)\geq 6\).  

Recall that \(G=(S_1\times S_2\times S_3)\times U\), where \(U=\bigcap_{i=1}^3\cent G{S_i}\); but, as \(\sigma(S_1\times S_2\times S_3)=5=\sigma(G)\), clearly the set \(\pi(U)-\pi(S_1\times S_2\times S_3)\) must be empty, and again by minimality we conclude that \(G=S_1\times S_2\times S_3\). At this stage, we see that \(|\pi(G)|\) is \(11\), so it equals \(2\sigma(G)+1\), and this is the contradiction which excludes \(|\mathcal{K}|=3\).

Finally, assume \(\mathcal{K}=\{S_1,S_2\}\) and write \(G=S_1\times S_2\times U\). Denoting by \(\chi_1^{(1)},\chi_2^{(1)} \) the characters of \(S_1\) given by Proposition~\ref{HTV} (as in the second paragraph of this proof) and, similarly, by \(\chi_1^{(2)},\chi_2^{(2)}\) those of \(S_2\), we consider the irreducible characters \(\chi_1=\chi_1^{(1)}\times\chi_1^{(2)}\) and \(\chi_2=\chi_2^{(1)}\times\chi_2^{(2)}\) of \(S_1\times S_2\). Note that, up to swapping \(\chi_1^{(2)}\) and \(\chi_2^{(2)}\), we may assume \(\{2,3\}\subseteq\pi(\chi_1)\). In this setting, we have 
\begin{equation}\label{eq2}
\pi(G)=\pi(U)\cup(\pi(\chi_1)-\pi(U))\cup(\pi(\chi_2)-\pi(U))\cup\{p_1,p_2\},
\end{equation}
and \(\sigma(G)\geq\sigma(U)+|\pi(\chi_i)-\pi(U)|\) for both \(i=1\) and \(i=2\).

Let us consider first the case \(U\neq 1\). Since \(\fit U=1\), any minimal normal subgroup of \(U\) is a direct product of isomorphic non-abelian simple groups, and the number of factors in this product has to be larger than \(1\); in fact, \(U\) being a direct factor of \(G\), a normal subgroup of \(U\) is also normal in \(G\), thus a simple normal subgroup of \(U\) would be an element of \(\mathcal{K}\), clearly a contradiction. As a consequence, we have \(\sigma(U)\geq 2\); on the other hand \(\sigma(U)\leq 3\), as otherwise, taking \(\delta\in\irr U\) such that  \( \pi(\delta) = \sigma(U)\) and  \(\psi\in\irr{S_1\times S_2}\) with \(\pi(\psi)\supseteq\{p_1,p_2\}\), we would get \(|\pi(\psi\times\delta)|\geq 6\).

If \(\sigma(U)=2\), by \({\bf{(vii)}}\) we have \(|\pi(U)|\leq 4\), and therefore Equation~(3) yields that both \(|\pi(\chi_1)-\pi(U)|\) and \(|\pi(\chi_2)-\pi(U)|\) must be \(3\) (otherwise we would get the contradiction \(|\pi(G)|<12\)). In particular there exist primes \(r,s,t\), all lying outside \(\pi(U)\), such that \(\pi(\chi_1)\supseteq\{2,3,r,s,t\}\). Now, as usual, we can produce an irreducible character of \(G\) whose degree is divisible by six primes taking the product of \(\chi_1\) and \(\xi\in\irr U\) with \(\xi(1)\) divisible by an odd prime different from \(3\).

Assume finally \(\sigma(U)=3\), so \(|\pi(U)|\leq 6\). By Equation~(3), we see that \(|\pi(U)|\) must be \(6\) and \(|\pi(\chi_i)-\pi(U)|\) is \(2\) for \(i\in\{1,2\}\). Now, there exist primes \(r,s\), both outside \(\pi(U)\), such that \(\pi(\chi_1)\supseteq\{2,3,r,s\}\); moreover, by the main result of \cite{MT}, we can find an irreducible character of \(U\) whose degree is divisible by two primes in \(\pi(U)-\{2,3\}\) (because this set contains four elements). Taking the product of this character with \(\chi_1\) we reach again the contradiction \(\sigma(G)\geq 6\).

To rule out the case \(|\mathcal{K}|=2\) and finish the proof, it remains to consider the case \(U=1\), so, \(G=S_1\times S_2\). In this situation, \(\cent G{S_1}=S_2\) is (almost-)simple, against what observed in {\bf{(iv)}}; thus also in this last case we reached a contradiction, and our minimal counterexample \(G\) does not exist.
\end{proof}
    
\begin{theorem}\label{morethan5}
Let $G$ be a group with trivial Fitting subgroup. If \(\sigma(G)> 5\), then we have \(|\pi(G)|\leq3\sigma(G)-4\).
\end{theorem}

\begin{proof}
As in the proof of Theorem~\ref{lessthan5}, we assume the existence of a counterexample to the statement, and we take $G$ of minimal order among these counterexamples. By Corollary~\ref{key2} and Theorem~\ref{TV} (taking also into account that \(2\sigma(G)+1\) is smaller than \(3\sigma(G)-4\) whenever \(\sigma(G)>5\)), the set of simple normal subgroups of \(G\) is non-empty and \(G\) is not an almost-simple group. So, let us consider a simple normal subgroup \(S\) of \(G\), and set \(C=\cent G S\) (note that \(C\neq 1\), because \(G\) is not almost-simple). We see that \(C\) has a trivial Fitting subgroup and clearly it is a proper subgroup of \(G\).

Recall the second paragraph in the proof of Theorem~\ref{lessthan5}: 
an application of Proposition~\ref{HTV} to the almost-simple group \(G/C\) yields that
there exist \(\chi_1\) and \(\chi_2\) in \(\irr {S}\) such that $\pi(S)=\pi(\chi_1) \cup \pi(\chi_2)$ unless \(S\) has an associated prime \(p\), in which case we have  $\pi(S)-(\pi(\chi_1) \cup \pi(\chi_2))\subseteq\{p\}$. Defining \(\pi_0=\pi(G/C)-\pi(S)\), $\pi_1 = \pi(\chi_1) \cup \pi_0$, $\pi_2 = \pi(\chi_2) \cup \pi_0$, and $\pi_3 =\{p\}$ or \(\pi_3=\emptyset\) according to whether \(S\) has an associated prime or not, we have
%By Clifford theory, for $i$ in $\{1,2,3\}$, if  $\gamma_i$ lies in  $\irr {G|\chi_i}$, then $\pi_i$ is contained in $\pi(\gamma_i)$; hence 
$\pi(G) = \pi(C) \cup \bigcup_{i=1}^3 (\pi_i - \pi(C))$, and therefore
\[  |\pi(G)| \leq |\pi(C)| + \sum_{i=1}^3 |\pi_i - \pi(C)|\leq|\pi(C)|+3\max_{i\in\{1,2,3\}}\{|\pi_i - \pi(C)|\}.
\]
%Hence, $\pi(G/C) \subseteq  \pi(\gamma_1)\cup \pi(\gamma_2)\cup \pi(\gamma_3)$.
(Note that \(|\pi_3-\pi(C)|\leq 1\).) Moreover, as seen in the previous theorem,
\[
\sigma(G) \geq  \sigma(C) + m,
\]
where \(m=\max_{i\in\{1,2,3\}}\{|\pi_i - \pi(C)|\}\).

Now, if $\sigma(C)>5$, then \(C\) satisfies our assumptions and, by minimality, we get $|\pi(C)|\leq 3\sigma(C)-4$. Thus \(|\pi(G)|\leq(3\sigma(C)-4)+3m=3(\sigma(C)+ m)-4\leq3\sigma(G)-4\), and \(G\) is not a counterexample.
   
Therefore we have $\sigma(C)\leq 5$. In this case, by Theorem~\ref{lessthan5}, we have  $|\pi(C)|\leq 2\sigma(C)+1$, hence $|\pi(G)|\leq (2\sigma(C)+1)+2 m+1\leq 2(\sigma(C)+ m)+2$. If $\sigma(C)+ m\leq 5$, then we also get \(\sigma(C)+ m\leq\sigma(G)-1\), whence \(|\pi(G)|\leq 2\sigma(G)\), which is impossible. On the other hand, \(\sigma(C)+ m\geq 6\) implies  \(2(\sigma(C)+ m)+2\leq 3(\sigma(C)+ m)-4\leq3\sigma(G)-4\), the final contradiction that completes the proof.
\end{proof}

\section{The solvable case: a proof of Theorem~B}

We begin now our discussion concerning the \(\rho\)-\(\sigma\) conjecture for solvable groups. After two general preliminary statements, we will prove a proposition on solvable permutation groups that will be a key step.

\begin{lemma}[\mbox{\cite[Lemma 3]{CD}}]\label{cd}
  Let $K \leq L \leq H$ be groups with $K \nor H$. Then, for every $S \leq H$, we have 
$$\pi(H:S) -\pi(H/K) = \pi(L:L\cap S) - \pi(H/K)\; .$$
\end{lemma}

\begin{lemma}\label{mn}
  Let $K$ be a group such that $K/\fit K$ is nilpotent. Then  there exists
  $\theta \in \irr K$ such that $\pi(K/\fit K )) \sbs \pi(\theta)$. 
\end{lemma}

\begin{proof}
  This follows from~\cite[Proposition 17.3]{MW}.
\end{proof}

\begin{proposition}\label{op2}
  Let $H$ be a solvable permutation group on a finite non-empty set $\Omega$.
  Then there exists a non-empty subset $\Delta \subseteq \Omega$ such that
  \begin{enumeratei}
  \item   $\pi(H) \sbs \pi(H:H_{\Delta}) \cup \{ p \}$, for a suitable $p \in \{ 2, 3\}$;
    \item   $\pi(H) =  \pi(H:H_{\Delta})$  if $|H|$ is odd. 
  \end{enumeratei}
\end{proposition}
\begin{proof}
  Part (b) is   part (b) of~\cite[Corollary 5.7]{MW}.
  
  In order to prove part (a), assume first that the action of $H$ on $\Omega$ is primitive. In this case we prove that there are two non-empty subsets $\Delta_1, \Delta_2 \sbs \Omega$  \emph{of distinct sizes} such that
  both set-stabilizers  $H_{\Delta_i}$ verify claim (a) in the statement, for the same prime $p$. 
  
  Clearly, if there exists a  subset $\Delta$ of $\Omega$ such that $H_{\Delta} = 1$ and $|\Delta|\neq |\Omega|/2$, then we set $\Delta_1 = \Delta$ and $\Delta_2 = \Omega - \Delta$.
  If no such $\Delta$ exists (in the terminology of \cite{MW}, this amounts to saying that there are no \emph{strongly regular} orbits in the action of \(H\) on the power set of \(\Omega\)), then by~\cite[Theorem 5.6]{MW} we are in one of the following cases.
  \begin{description}
 
 \item[(a)] $ |\Omega| = 2$: then we take any subset of size $1$ as $\Delta_1$ and $\Delta_2 = \Omega$.

\item[(b)] $|\Omega| = 3$, $H \cong S_3$:  take $|\Delta_1| =1 $, $|\Delta_2| = 2$.

\item[(c)] $|\Omega| = 4$, $H \cong A_4$ or $S_4$; take $|\Delta_1| = 1$, $|\Delta_2| = 3 $.

\item[(d)] $|\Omega| = 5$: $H$ is a Frobenius group of order $10$ or $20$; take $|\Delta_1| =3$, $|\Delta_2| = 4 $.

\item[(e)] $|\Omega| = 7$: $H$ is a Frobenius group  of order $42$; take $|\Delta_1| =2 $, $|\Delta_2| =5 $.

\item[(f)] $|\Omega| =8 $: $H \cong A\Gamma(2^3)$; take $|\Delta_1| =2 $, $|\Delta_2| = 6$.
  
\item[(g)] $|\Omega| = 9$ and $H$ is a subgroup of the semidirect product of ${\rm GL}_2(3)$ with its natural module $V$; here we take $|\Delta_1| =1 $, $|\Delta_2| = 8$ (note that $H$ is a
  $\{2,3\}$-group and that $V$ is a regular normal subgroup).
\end{description}

Next, we assume that $H$ is transitive but imprimitive on $\Omega$.
Let $\Omega = \Sigma_1 \cup \Sigma_2 \cup \cdots \cup \Sigma_m$ be an imprimitive decomposition of $\Omega$, with minimal $|\Sigma_i| \neq 1$.   
Let $\Sigma = \Sigma_1$ and write $\Sigma_i = \Sigma^{x_i}$ for suitable elements $x_1= 1, x_2, \ldots, x_m \in H$.
Let $L = H_{\Sigma}$ be the stabilizer of $\Sigma$ in $H$ and let
$K = \bigcap_{i=1}^m L^{x_i}$ be the kernel of the action of $H$ on the set $\overline{\Omega} = \{ \Sigma_1, \ldots, \Sigma_m\}$.
  So $\overline{H} = H/K$ is a solvable (transitive)  permutation group on  $\overline{\Omega}$, and by~\cite[Corollary 3]{D} (or by Proposition~\ref{permutation}) there exist two disjoint subsets $\Gamma_1, \Gamma_2$ of
  $\overline{\Omega}$ such that $\overline{H}_{\Gamma_1} \cap \overline{H}_{\Gamma_2}$ contains no Sylow subgroup of $\overline{H}$.
%  Up to renumbering, we can assume that $\Gamma_1 = \{\Sigma_1, \ldots_
Let $C$ be the kernel of the action of $L$ on $\Sigma$; so $L/C$ is a \emph{primitive} permutation group on $\Sigma$ (by the minimality of $|\Sigma|$) and there exist two non-empty subsets $\Delta_1, \Delta_2 \sbs \Sigma$
such that $|\Delta_1| \neq |\Delta_2|$ and $ \pi(L:L_{\Delta_i}) \supseteq \pi(L/C) - \{p\}$, $i=1,2$, with $p \in \{ 2, 3 \}$.
We can also assume that $\Sigma_1 \in \Gamma_1$.

Define now
$$\Delta = \bigcup_{\Sigma^{x_i} \in \Gamma_1} \Delta_1^{x_i} \cup \bigcup_{\Sigma^{x_i} \in \Gamma_2} \Delta_2^{x_i} \; .$$
%So $\Delta$ is a non-empty subset of $\Omega$.
Since $|\Delta_1| \neq |\Delta_2|$, it follows that $H_{\Delta}K/K = \overline{H_{\Delta}} \leq \overline{H}_{\Gamma_1} \cap \overline{H}_{\Gamma_2}$. 
Hence $\pi(H/K) \sbs \pi(H:H_{\Delta})$.
Moreover,   by Lemma~\ref{cd}, $\pi(H:H_{\Delta}) - \pi(H/K) = \pi(L:L_{\Delta}) - \pi(H/K)$.
As $L_{\Delta}$ is a subgroup of  $L_{\Delta_1}$ (since $\Sigma_1 \in \Gamma_1$)  and $\pi(K) \subseteq \pi(L/C)$ (because $K$ is isomorphic to a subgroup of a direct product of groups isomorphic to $L/C$),  we deduce that
$\pi(H)  = \pi(K) \cup \pi(H/K) \sbs  \pi(H:H_{\Delta}) \cup \{p\}$.

Finally, we assume that $H$ is not  transitive on $\Omega$; let $\Omega_1, \Omega_2, \ldots, \Omega_k$, $k \geq 2$, be the orbits of $H$ on $\Omega$ and let
$K_i$ be the kernel of the action of $H$ on $\Omega_i$, for $i = 1, 2, \ldots, k$.
By what we have proved till now, for each $i$ there exists a non-empty $\Delta_i \sbs \Omega_i$  and a prime $p_i \in \{2,3\}$  such that
$\pi(H/K_i) \sbs \pi(H/K_i:(H/K_i)_{\Delta_i}) \cup \{ p_i \} =  \pi(H:H_{\Delta_i}) \cup \{p_i\}$.
By part (b), if $2 \not\in \pi(H/K_i)$, then  we can choose $\Delta_i \sbs \Omega _i$ such that $\pi(H/K_i) = \pi(H:H_{\Delta_i})$.

As $\bigcap_i K_i = 1$, we have $\pi(H) = \bigcup_i \pi(H/K_i)$.
Hence, setting $\Delta = \bigcup_i \Delta_i$, we conclude that $\pi(H) \sbs \pi(H:H_{\Delta}) \cup \{p\}$, where $p = 2$ if $p_i = 2$ for each $i$  and $p=3$ if  $p_i=3$ for some $i$.
\end{proof}
%%%%%%%%%%%%%%%%%%%%%%%%%%%%%%%%%%%

Next, we introduce some terminology. If $H$ is a group and $V$ is an irreducible $H$-module, then it is possible to find a pair \((L,W)\), where \(L\) is a subgroup of \(H\) and \(W\) is a primitive \(L\)-module, such that we have $V = W^H$ (here \(L=\norm H W\)); if $V$ is primitive, then $L = H$. We say that a faithful irreducible \(H\)-module \(V\) is of \emph{type 1} if there exists a pair \((L,W)\) of this kind, such that the normal core \(L_H\) of \(L\) in \(H\) is metabelian.
 % If $L/\cent LW \leq \Gamma(W)$ (semilinear group on $W$; this happens exactly  when  $\fit{\o L}$ is abelian),
 
  In the situation described in the previous paragraph, setting  $R = L/\cent LW$, we have $\fit R = E Z$ where $Z = \zent{\fit R}$ is cyclic, $E \nor R$  and  $E$  is a
 direct product of extraspecial groups. Then one sets  $e_W = \sqrt{|E/\zent E|}$ (\cite[Definition 2.1]{Y1}).
 We remark that  $e_W = 1$ if and only if $\fit R$ is cyclic and $R$ acts as a group of semilinear maps on $W$. 
Hence, if $e_W = 1$, then $L_H$ is metabelian (because it embeds in the direct product of  groups  isomorphic to the metabelian group $L/\cent LW$) and \(V\) is of type 1.

\begin{proposition}\label{p3}
  Let $G$ be a solvable group whose Frattini subgroup $\frat G$ is trivial. Then there  exist $\chi, \psi \in \irr G$ and a prime  $p \in \{ 2, 3\}$ such that
  $\rho(G) \sbs  \pi(\chi) \cup \pi(\psi) \cup \{p\}$.
  Furthermore, if $|G/\fit G|$ is odd, then there exist $\chi, \psi \in \irr G$ such that
  $\rho(G) = \pi(\chi) \cup \pi(\psi)$.
\end{proposition}
\begin{proof}
  As $\frat G$ is trivial, $F = \fit G$ is a direct product of minimal normal subgroups of $G$ and it has a complement $H$ in $G$ (see \cite[III.4.4 and III.4.5]{H}).
  Write $F = M_1 \times \cdots \times M_n$, where the $M_i$ are minimal normal subgroups of $G$, and let $V_i$ be the dual group $\widehat{M_i}$ for \(i\in\{1,\ldots,n\}\).
  So $V_i$ is an irreducible $H$-module, and
  $V=\widehat{F} = V_1 \times \cdots \times V_n$ is a completely reducible and faithful $H$-module. We will work by induction on $|G|$.

  \medskip
 We start by assuming that \(V=V_1\) is an irreducible \(H\)-module (of course faithful), and we write $V = W^H$, where $W$ is a primitive $L$-module for $L = \norm HW$.
  Let $K = L_H$ be the normal core of $L$ in $H$; then $\o H = H/K$ is a transitive permutation group on the set $\Omega = \{ W^{x_1}, \ldots, W^{x_m}\}$, where
  $\{ x_1 = 1, \ldots, x_m\}$ is a right transversal of $L$ in $H$.
  By Proposition~\ref{op2}, there exists a non-empty subset $\Delta$ of $\Omega$ and a prime  $p \in \{ 2, 3\}$ such that $\pi(\o H : \o{H}_{\Delta}) \supseteq \pi(\o H) - \{p\}$.
  Moreover,  we can assume that either $\pi(\o H : \o{H}_{\Delta})  = \pi(\o H)$  or $2\in \pi(\o{H})$, and there is no loss of generality in assuming  $W = W_1 \in \Delta$. 
  
  Now, take any set \(\{\mu_1,\ldots,\mu_m\}\) where \(\mu_i\) lies in \(W^{x_i}\) for \(i\in\{1,\ldots,m\}\) and \(\mu_i\) is non-trivial if and only if \(W^{x_i}\in\Delta\). Define then $\lambda = \mu_1 \times \cdots \times \mu_m \in V$. Taking into account the imprimitivity decomposition of $V$ and the choice of the $\mu_i$, setting  $I = \cent H{\lambda}$ we see that $\o I = IK/K$ is contained in the stabilizer ${\o H}_{\Delta}$, whence $\pi(H:I) \supseteq \pi(H/K) -\{p\}$.

  \smallskip
Consider first the case when $V$ is an \(H\)-module of type 1, so (for a suitable choice of the pair $(L, W)$) $K$ is metabelian, and let \(\lambda\in V\) be as above. Since $\lambda$ is not the trivial character of \(F\) (because \(\Delta\neq\emptyset\)), we claim that $\pi(H:I) \supseteq \pi(\fit{K})$ as well. In fact, if $q \in \pi(\fit K)$ and $Q$ is a  Sylow $q$-subgroup of $\fit K$, then $1 \neq Q \nor H$ and hence (by Clifford Theorem) $V_Q$ has no trivial irreducible constituent; in particular, \(\langle \lambda\rangle\) can not be a trivial \(Q\)-module, hence \(Q\) does not lie in \(I\). Thus, 
  for any $\chi \in \irr{G|\lambda}$, we deduce that  $\pi(\chi) \supseteq (\pi(H/K) \cup \pi(\fit K) ) - \{ p \}$.
  Moreover, by Lemma~\ref{mn}  there exists  $\theta \in \irr K$ such that $\pi(K/\fit K )) \sbs \pi(\theta)$. So, taking a character $\psi \in \irr{H|\theta}$ and viewing $\psi$ as
  an irreducible character of $G$ by inflation, we get
  $\rho(G) = \pi(H) \sbs  \pi(\chi) \cup \pi(\psi) \cup \{p\}$, where $p \in \{ 2, 3\}$ and
     $\rho(G) = \pi(H) =  \pi(\chi) \cup \pi(\psi)$ if $|H|$ is odd.   

   \smallskip
   Assume next that $V$ is irreducible and not of type 1 as an \(H\)-module, thus $e = e_W \geq 2$.
   Here we aim to prove a stronger property: there
  exists $\lambda \in V$ such that,  for all $\chi \in \irr{G|\lambda}$,
  we have $\rho(G) = \pi(\chi)  \cup \{p\}$ where $p \in \{ 2, 3\}$.

   By Theorem 3.1 of~\cite{Y1} and Theorem 3.1 of~\cite{Y2}, either $R=L/\cent L W$ has a regular orbit on $W$ or $e$ belongs to the set $\{1, 2, 3, 4, 8, 9, 16\}$. In the latter case, we see that the center $\zent E$ of $E$ is a non-trivial $q$-group where $q \in \{2,3\}$, and $\zent E$ acts fixed-point freely on $W$.
   Moreover, by Theorem 11.3 of~\cite{MW} there exists a (non-trivial) $\mu \in W$ such that  $\pi(R:\cent R{\mu}) \supseteq \pi(R) - \{2,3\}$. As $\zent E$ does not fix $\mu$, we
   conclude that $\pi(R:\cent R{\mu}) \supseteq \pi(R) - \{p_0\}$ for a suitable  $p_0 \in  \{ 2,3 \}$.
  In fact we remark that, if $\pi(R:\cent R{\mu})$ does not contain $\pi(R) - \{ 2 \}$, then \(3\) does not divide \(|\zent E|\) and so $\fit R$ is a central product of an extraspecial $2$-group and a cyclic group. 
   Finally, if  $2 \not\in \pi(R)$ then by~\cite[Theorem 2.2]{D1}  we can assume that $\pi(R:\cent R{\mu}) = \pi(R)$. Note that the last equality clearly holds also if \(R\) does have a regular orbit on \(W\), provided \(\mu\) is chosen in such an orbit.
   
   Define $\lambda = \nu_1 \times \cdots \times \nu_m \in V$, where  $\nu_i = \mu^{x_i}$ if $W^{x_i} \in \Delta$ and $\nu_i $ is trivial otherwise. Set also $I = \cent H{\lambda}$, as above.
   Then we already know that $\pi(H:I)\supseteq \pi(H/K)- \{p\}$. 
   On the other hand, since  $L \cap I = \cent L{\lambda} \leq \cent L{\mu}$, and since we have $\pi(R)\supseteq \pi(K)$ because \(K\) embeds in the direct product of groups isomorphic to \(R\), Lemma~\ref{cd} yields
   $$\pi(H:I) - \pi(H/K) = \pi(L: L \cap I) - \pi(H/K) \supseteq \pi(R) - \left(\pi(H/K) \cup \{p_0\}\right) \supseteq \pi(K) - \left( \pi(H/K) \cup  \{p_0\}\right) .$$  
Let now \(\chi\) be in \(\irr{G|\lambda}\), and observe first that if \(p_0=p\), then we clearly have \(\rho(G)\subseteq\pi(\chi)\cup\{p_0\}\). Assume then \(p_0=2\) and \(p=3\). If \(\Delta\) can be chosen so that $\pi(H:I)\supseteq \pi(H/K)$, then we get \(\rho(G)\subseteq\pi(\chi)\cup\{2\}\); on the other hand, if such a \(\Delta\) does not exist, then \(2\) lies in \(\pi(H/K)\) and we have \(\rho(G)\subseteq \pi(\chi)\cup\{3\}\). 
The only case that still needs to be treated is when $3$ divides \(|K|\) but not \(|R:\cent R{\mu}|\) (so, \(p_0=3\) and \(p=2\)). In this setting, \(\pi(R:\cent R{\mu})\) does not contain \(\pi(R)-\{2\}\), and therefore (as observed in the paragraph above) the \(2\)-complement of \(\fit R\) is cyclic; we claim that such a situation forces \(2\) to lie in \(\pi(K)\). In fact, assuming the contrary and setting \(C=\cent L W\), we see that \(\fit{KC/C}\) is contained  in the \(2\)-complement of \(\fit R\) and it is therefore cyclic. It follows that the factor group of \(KC/C\) over its Fitting subgroup (which embeds in the automorphism group of a cyclic group of odd order) is cyclic as well, so that \(KC/C\) is metabelian. But this contradicts the fact that \(V\) is not of type 1, because \(K\) now embeds in the direct product of groups isomorphic to \(KC/C\) and it is then metabelian as well. This contradiction yields \(2\in\pi(H:I)\), hence we have \(\rho(G)=\pi(\chi)\cup\{3\}\) for any \(\chi\in\irr{G|\lambda}\).
   
 Finally, if $|H|$ is odd, then  $\pi(\o H : \o{H}_{\Delta}) = \pi(\o H)$ and  $\pi(R:\cent R{\mu}) = \pi(R)$, so we get $\rho(G)  =  \pi(\chi)$ for any \(\chi\in\irr{G|\lambda}\). 
   
   \medskip
It remains to consider the case when \(V\) is not irreducible, so, we assume $n \geq 2$.
Suppose first that one of the modules $V_i$, say $V_1$, is not of type 1 as an \(H/\cent H{V_1}\)-module. Let  $L$ be a complement of $V_1$ in $G$ and $C = \cent L{V_1} \nor G$. Write $U = \fit C = \fit G \cap C$. Applying to the \(L/C\)-module \(V_1\) (not of type 1) the discussion carried out above in the irreducible case,  
   we can find $\lambda_0 \in V_1$ such that,  for all $\theta \in \irr{G|\lambda_0 \times 1_C}$, we have $\pi(L/C) \sbs \pi(\theta)  \cup \{p_0\}$
   for a suitable $p_0 \in \{ 2, 3\}$. Furthermore, if $|L/C|$ is odd, then $\pi(L/C) = \pi(\theta)$.

   Since $\frat C = 1$, by inductive hypothesis there exist $\alpha, \beta \in \irr C$ such that $\pi(C/U) = \rho(C) \sbs  \pi(\alpha) \cup \pi(\beta) \cup \{p_1\}$,
   for a suitable $p_1 \in \{ 2, 3\}$; and $\pi(C/U) = \pi(\alpha) \cup \pi(\beta)$ if $|C/U|$ is odd. 
   Let now  $\chi \in \irr{G|\lambda_0 \times \alpha }$  and  $\psi \in \irr{G|\lambda_0 \times \beta }$. 
Observe that  $\rho(G) = \pi(L/C) \cup \pi(C/U)$, that if $\pi(L/C)-\pi(\theta)=\{3\}$ then $2 \in \pi(L/C)$ and if $\pi(C/U)-(\pi(\alpha)\cup\pi(\beta))=\{3\}$ then $2 \in \pi(C/U)$. Hence, setting $ p = \max \{p_0, p_1\}$, we 
have $\rho(G) \sbs \pi(\chi) \cup \pi(\psi) \cup \{ p \}$ with $p \in \{2,3 \}$. Moreover, if $|G/\fit G| = |L/U|$ is odd, then 
$\rho(G) \sbs \pi(\chi) \cup \pi(\psi)$.

   Therefore, we can assume that each of the \(V_i\) is an irreducible \(H/\cent H{V_i}\)-module of type 1. Write $V_i = (W_i)^H$, where $W_i$ is a primitive $L_i$-module for a subgroup $L_i$ of $H$ and, setting $K_i = (L_i)_H$, the factor group \(K_i/\cent H{V_i}\) is metabelian. Let $K = \bigcap_i K_i$. Then $\o H = H/K$ is a permutation group  on the set $\Omega = \bigcup_i \Omega_i$, where
   $\Omega_i$ is the set consisting of the conjugates of the module $W_i$ by the action of $H$ (so the sets $\Omega_i$ are the orbits of $\o H$ on $\Omega$).
   By Proposition~\ref{op2}, we can choose a (non-empty) subset $\Delta \sbs \Omega$,  such that 
   $\pi(\o H) \sbs \pi(\o H :\o{H}_{\Delta}) \cup \{ p \}$, where $p \in \{ 2, 3\}$, and
     $\pi(\o H) \sbs \pi(\o H :\o{H}_{\Delta})$, if $|\o H|$ is odd. We can also clearly assume that $\Delta$ a has non-empty intersection with every orbit $\Omega_i$.  

  For every $W \in \Omega$, we now choose a  $\mu_W \in W$ such that $\mu_W \neq 1_W$ if and only if $W \in \Delta$.
  We define $\lambda = \prod_{W \in \Omega}\mu_W \in V$ and $I = \cent H{\lambda}$.  
  Arguing as in the irreducible case of type~1,
  we conclude that, for every $\chi \in \irr{G|\lambda}$,   $\pi(\chi) \supseteq (\pi(H/K) \cup \pi(\fit K) ) - \{ p \}$, and
  $\pi(\chi) \supseteq (\pi(H/K) \cup \pi(\fit K) )$ if $|H/K|$ is odd.
  Since $K$ is metabelian (as every group $K_i/\cent H{V_i}$ is metabelian and $\bigcap_i \cent H{V_i} = 1$), by Lemma~\ref{mn}  there exists  $\theta \in \irr K$ such that $\pi(K/\fit K )) \sbs \pi(\theta)$. So, taking a character $\psi \in \irr{H|\theta}$ and viewing $\psi$ as
  an irreducible character of $G$ by inflation, we get that
   $\rho(G) = \pi(H) \sbs  \pi(\chi) \cup \pi(\psi) \cup \{p\}$, where $p \in \{ 2, 3\}$, and that $\rho(G) = \pi(\chi) \cup \pi(\psi)$ if $|H|$ is odd.  
 \end{proof}

We are ready to prove a result that will imply Theorem~B.

 \begin{theorem}\label{KRS}
   If \(G\) is a  solvable group, then there exist
   $\beta_1, \beta_2, \beta_3 \in\irr G$ such that
   % $\(\beta_1,\beta_2,\beta_3\) of \(G\) that yield the bound
   \[|\rho(G)| \leq|\pi(\beta_1)|+|\pi(\beta_2)|+|\pi(\beta_3)|\]
   and
   \[\rho(G)=\pi(\beta_1)\cup\pi(\beta_2)\cup\pi(\beta_3)\cup\{p\}\] for a suitable \(p\in\{2,3\}\).
 
 \end{theorem}
 \begin{proof}
   Set $\o G = G/\frat G$; by Proposition~\ref{p3} there exist (by inflation) $\beta_1, \beta_2 \in \irr G$     and a suitable $p \in \{ 2, 3\}$ such that,
   setting $\pi_i = \pi(\beta_i)$ for $i=1,2$, we have $\pi(G/\fit G) = \rho(\o G) \sbs \pi_1 \cup \pi_2 \cup \{p\}$. Define also \(\nu\) as the set of primes in \(\pi(G)\) for which \(G\) has a normal non-abelian Sylow subgroup,
   % $$\nu = \{ q \in \pi(G) \mid \exists \; Q \in \syl qG \text{ with } Q \nor G \text{ and } Q \text{ non-abelian} \} \; ,$$
    so that $\rho(G) = \nu \cup \rho(\o G)$, and let $N$ be a Hall $\nu$-subgroup of $G$. Now,
   % note that $N \leq \fit G$ and that
   $N$ is a normal subgroup of $G$ and there exists $\varphi \in \irr N$ such that $\pi(\varphi) = \nu$.
   Clearly $|\pi_1|, |\pi_2|, |\nu| \leq \sigma(G)$  and 
   $$\rho(G) = \nu \cup \rho(\o G)  \sbs \nu \cup \pi_1 \cup \pi_2  \cup \{p\} \; .$$
   
   If $|\pi_1 \cap \pi_2| \geq 1$, then $|\pi_1 \cup \pi_2 \cup \{ p\}| \leq  |\pi_1 \cup \pi_2| +1 \leq |\pi_1| + |\pi_2|$ and hence
   $$|\rho(G)|\leq |\nu| + |\pi_1| + |\pi_2|  \leq 3 \max \{|\nu|,  |\pi_1|, |\pi_2|\} \leq 3 \sigma(G) \; .$$
   Hence we can assume that $\pi_1$ and $\pi_2$ are disjoint sets. By a similar argument,  we can assume   $\nu  \cap (\pi_1 \cup \pi_2 \cup \{p\})   = \emptyset$. 
   
Assume now that,  for every $\chi \in \irr{G|\varphi}$ (so, $\nu \sbs \pi(\chi)$), we have $\pi(\chi) = \nu$.
Since $(|N|, |G/N|) = 1$, we have that $\varphi$ is $G$-invariant (by Clifford correspondence) and  there exists an extension
   $\theta \in \irr G$ of $\varphi$. 
   Thus Gallagher  theorem yields that $G/N$ is abelian and result easily follows  by Lemma~\ref{mn}. 

   Therefore we can assume that there exists $\beta_3 \in \irr{G|\varphi}$  such that $\nu$ is a proper subset of  $\pi(\beta_3)$.
   Hence, $\pi(\beta_3) \cap (\pi_1 \cup \pi_2 \cup \{ p\}) \neq \emptyset$ and then, arguing along the same line as above, we conclude that $|\rho(G)| \leq|\pi(\beta_1)|+|\pi(\beta_2)|+|\pi(\beta_3)|$.
 \end{proof}

From the first inequality Theorem~\ref{KRS} we hence deduce the following corollary, which is Theorem~B of the Introduction.
 
 \begin{corollary}
  If \(G\) is solvable group, then $|\rho(G)| \leq 3\sigma(G)$. 
 \end{corollary}

\section{A generalization of the solvable case}

We conclude our discussion by proving Theorem~C, which extends the inequality \(\rho(G)\leq 3\sigma(G)\) to all groups \(G\) such that \(\cent G{\E G}\) is solvable. As mentioned in the Introduction, this class of groups includes both the solvable groups and the groups with trivial Fitting subgroup.

First, we need an application of Theorem~\ref{key} which takes also into account Proposition~\ref{HTV}.

\begin{theorem}\label{keyHTV}
Let \(G\) be a group with \(\fit G=1\). Then there exist \(\alpha_1,\alpha_2,\alpha_3\in\irr{\fitg{G}}\) (not necessarily distinct) such that, setting \(m=\m{G/\fitg G}\), %for any choice of \(\theta_i\in\irr{G|\alpha_i}\) 
the following conclusions hold.
\begin{enumeratei}
\item \(\pi_{\geq m}(G)\subseteq \bigcup^2_{i=1} \big(\pi(\alpha_i)\cup \pi(G:I_G(\alpha_i))\big)\cup\pi(\alpha_3)\).
%\item \(\pi_{\geq m}(G)\subseteq \pi(\theta_1)\cup\pi(\theta_2)\cup\pi(\theta_3)\);
\item \(|\pr_m|\leq 2|\pi(G:I_G(\alpha_i))\cap\pr_m|\) for \(i=1,2\).
%\item \(|\pr_m|\leq 2|\pi(\theta_i)\cap\pr_m|\) for \(i=1,2\).
\end{enumeratei}
\end{theorem}

\begin{proof} Let us denote by \(M\) the subgroup generated by all the simple characteristic subgroups of \(G\), so that \(M=S_1\times\cdots\times S_t\) where the \(S_j\) are non-abelian simple groups, and set \(C=\cent G M\). Since \(\fit C=1\) and clearly \(C\) does not have any simple characteristic subgroup, we can apply Theorem~\ref{key} to the group \(C\): defining \(m_0=\m{C/\fitg C}\), there exist two irreducible characters \(\psi_1,\psi_2\in\irr{\fitg C}\) such that \[\pi_{\geq m_0}(C)\subseteq \bigcup^2_{i=1} \pi(\psi_i)\cup \pi(C:I_C(\psi_i))\quad {\textnormal{and}}\quad |\pr_{m_0}|\leq 2|\pi(C:I_C(\psi_i))\cap\pr_{m_0}|\]  for $i=1,2$.

Next we use Proposition~\ref{HTV} to get, for all \(j\in\{1,\ldots,t\}\), three irreducible characters \(\chi_1^{(j)},\chi_2^{(j)},\xi^{(j)}\) of \(S_j\) such that \[\pi(S_j)\subseteq\pi(\chi_1^{(j)})\cup\pi(\chi_2^{(j)})\cup\pi(\xi^{(j)})\quad {\textnormal{and}}\quad\pi(G/\cent G{S_j})-\pi(S_j)\subseteq\pi(G:I_G(\chi_i^{(j)})),\;i=1,2.\]

Consider now the irreducible characters \(\chi_i=\chi_i^{(1)}\times\cdots\times\chi_i^{(t)}\) (\(i\in\{1,2\}\)) and \(\xi=\xi^{(1)}\times\cdots\times\xi^{(t)}\) of \(M\); then, observing that \(\fitg G=M\times\fitg C\), define \(\alpha_1=\chi_1\times\psi_1\), \(\alpha_2=\chi_2\times\psi_2\), \(\alpha_3=\xi\times 1_{\fitg{C}}\). It is routine to check that \(\alpha_1,\alpha_2,\alpha_3\in\irr{\fitg G}\) satisfy (a); moreover, taking into account that either \(m=m_0\) or \(\pr_{m_0}\) is contained in \(\pi(G/MC)\), the characters \(\alpha_1\) and \(\alpha_2\) are easily verified to satisfy (b) as well.
\end{proof}

Finally, we prove Theorem~C.

 \begin{ThmC}
  Let $G$ be a group such that $\cent G{\E G}$ is solvable. Then $|\rho(G)| \leq 3\sigma(G)$. 
 \end{ThmC}
 \begin{proof}
   Set $ E = {\bf E}(G)$ and $F = \fit G$.
   We know (see~\cite[6.5.2]{KS} ) that every solvable normal subgroup of $G$ centralizes $ E$. Hence, by assumption we have $\cent GE = R$,
   where $R$ is  the solvable radical of $G$. Let $Z = E \cap F = \zent E$.
   
 Our first claim is that \(\rho(G/Z)=\rho(G)\). So, take \(p\in\rho(G)\) and denote by \(P\) a Sylow \(p\)-subgroup of \(G\). We have $Z \leq  \zent E \cap E' \leq \frat E$, so $Z \leq \frat G$ and hence $F/Z = \fit{G/Z}$; it follows that if \(P\) is not normal in \(G\), then \(PZ/Z\) is not normal in \(G/Z\) and hence \(p\) lies in \(\rho(G/Z)\). On the other hand, if \(P\) is normal in \(G\) (and non-abelian), then \(E\cap P\) is contained in \(E\cap F= Z\), whence \(p\) does not divide \(|E/Z|\). As a consequence, \(p\) does not divide the order of the Schur multiplier \({\rm{M}}\) of \(E/Z\) and, in particular, \(p\nmid|Z|\) (in fact, note that \(E\) is a central extension of \(E/Z\) whose kernel \(Z\) lies in \(E'\), so \(Z\) embeds in \({\rm{M}}\) by \cite[Corollary 11.20(b)]{Is}). We conclude that \(PZ/Z\simeq P\) is non-abelian, so \(p\in\rho(G/Z)\). 
 
Furthermore, by Lemma~3.2 of~\cite{CDPS} we have $E/Z\subseteq\E{G/Z}$ (it can be  easily checked that actually equality holds) and $\cent{G/Z}{E/Z} = R/Z$; it follows that \(\cent{G/Z}{\E{G/Z}}\subseteq R/Z\) is solvable. Our conclusion after the last two paragraphs is  that we can assume \(Z=1\).
   
Note that now $Z = E \cap R=1$, so $ER = E \times R$. An application of Theorem~\ref{keyHTV} to the factor group \(G/R\) (whose Fitting subgroup is clearly trivial, and whose generalized Fitting subgroup is \(ER/R\)) yields the existence of \(\alpha_1, \alpha_2,\alpha_3\in\irr E\) such that, setting \(\pi_i=\pi(\alpha_i)\cup \pi(G:I_G(\alpha_i))\) for \(i\in\{1,2\}\) and \(\pi_3=\pi(\alpha_3)\), we have
 \[\pi_{\geq m}(G/R)\subseteq \pi_1\cup\pi_2\cup\pi_3\quad{\textnormal{and}}\quad |\pr_m|\leq 2|\pi(G:I_G(\alpha_i))\cap\pr_m|\quad{\textnormal{for }} i=1,2\]
 where \(m=\m{G/ER}\). Observe that, in general, the sets \(\pi_1\), \(\pi_2\) and \(\pi_3\)  do not cover together the whole \(\pi(G/R)\), but they clearly do so in the case \(m=0\); at any rate, it is not difficult to see that the two displayed formulas above imply \(|\pi({G/R})|\leq |\pi_1|+|\pi_2|+|\pi_3|\). 
 
Also, consider the characters \(\beta_1,\beta_2,\beta_3\in\irr R\) provided by an application of Theorem~\ref{KRS} to \(R\), such that \[\rho(R)\subseteq \pi(\beta_1)\cup\pi(\beta_2)\cup\pi(\beta_3)\cup\{p\}\quad{\textnormal{and}}\quad |\rho(R)|\leq|\pi(\beta_1)|+|\pi(\beta_2)|+|\pi(\beta_3)| .\] 
Note that $p$ lies in $\mathbb{P}_m$ if $m \neq 0$, as then $m \geq 5$. 

We are now ready to finish our proof. If $m \neq 0$,  we take $\alpha = \alpha_i$ such that $|\pi_i| = \max_{j=1,2,3}\{|\pi_j|\}$, and $\beta = \beta_i$ such that
  $|\pi(\beta) - \pi(G/R)| = \max_{j=1,2,3}\{|\pi(\beta_j) - \pi(G/R)|\}$.
If $m = 0$,  we take $\beta = \beta_i$ such that  $|\pi(\beta)| = \max_{j=1,2,3}\{|\pi(\beta_j)|\}$ and $\alpha = \alpha_i$ such that 
$|\pi_i - \pi(R)| = \max_{j=1,2,3}\{|\pi_j -\pi(R)|\}$.
One easily sees that, for any $\chi\in \irr{G|\alpha\times \beta}$,  in both cases we have
$|\rho(G)| \leq 3\pi(\chi) \leq 3\sigma(G).$  
 \end{proof}

{\bf Acknowledgment.} Some part of this research has been carried out during a visit of the first author at
the Dipartimento di Matematica e Informatica ``U. Dini" (DIMAI) of the University
of Firenze. She wishes to thank the DIMAI for the hospitality.

%%%%%%%%%%%%%%%%%%%%%%%%%%%%%%%%%%%%%%%%%%%%%%%%%%%%%%%%%%%%%%%%%%%%%%%%%%%%%%%%%%%%%%%%%%%%%%%%%%%%%%%%%%%%%%%%%
%%%%%%%%%%%%%%%%%%%%%%%%%%%%%%%%%%%%%%%%%%%%%%%%%%%%%%%%%%%%%%%%%%%%%%%%%%%%%%%%%%%%%%%%%%%%%%%%%%%%%%%%%%%%%%%%%
%%%%%%%%%%%%%%%%%%%%%%%%%%%%%%%%%%%%%%%%%%%%%%%%%%%%%%%%%%%%%%%%%%%%%%%%%%%%%%%%%%%%%%%%%%%%%%%%%%%%%%%%%%%%%%%%%

\end{document}